%% file: topologies_finies.tex
\documentclass{amsart}
\usepackage{amscd}
\usepackage{amsmath}
\usepackage{amssymb}
\usepackage{latexsym}
\usepackage[latin1]{inputenc}
\usepackage[T1]{fontenc}   
\usepackage[english]{babel}

\input{arbres}

\input{posets}

\input{xy}
\xyoption{all}

\title{The Hopf algebra of finite topologies and $T$-partitions}

\author{Lo\"\i c Foissy}
\address{Fédération de Recherche Mathématique du Nord Pas de Calais FR 2956\\
Laboratoire de Mathématiques Pures et Appliquées Joseph Liouville\\
Université du Littoral Côte d'opale\\ Centre Universitaire de la Mi-Voix\\ 50, rue Ferdinand Buisson, CS 80699\\ 62228 Calais Cedex, France\\
email: foissy@lmpa.univ-littoral.fr}

\author{Claudia Malvenuto}
\address{Dipartimento di Matematica\\ Sapienza Universit\`a di Roma\\ P.le A. Moro 2\\ 00185, Roma, Italy\\
email: claudia@mat.uniroma1.it}

\newtheorem{defi}{\indent Definition}
\newtheorem{lemma}[defi]{\indent Lemma}
\newtheorem{cor}[defi]{\indent Corollary}
\newtheorem{theo}[defi]{\indent Theorem}
\newtheorem{prop}[defi]{\indent Proposition}

\newcommand{\FQSym}{\mathbf{FQSym}}
\newcommand{\h}{\mathbf{H}}
\newcommand{\K}{\mathbb{K}}
\renewcommand{\L}{\mathcal{L}}
\newcommand{\N}{\mathbb{N}}
\renewcommand{\P}{\mathcal{P}}
\renewcommand{\SS}{\mathcal{S}}
\newcommand{\T}{\mathcal{T}}
\newcommand{\TT}{\mathbf{T}}
\renewcommand{\S}{\mathfrak{S}}
\newcommand{\SP}{\mathbf{SP}}
\newcommand{\QSym}{\mathbf{QSym}}
\newcommand{\WQSym}{\mathbf{WQSym}}
\newcommand{\notsim}{\sim \hspace{-2mm} \mid\hspace{0mm}_\T \hspace{1mm}}
\def\shuff#1#2{\mathbin{\hbox{\vbox{\hbox{\vrule \hskip#2 \vrule height#1 width 0pt}\hrule}\vbox{\hbox{\vrule \hskip#2 \vrule height#1 width 0pt\vrule }\hrule}}}}
\def\shuffl{{\mathchoice{\shuff{5pt}{3.5pt}}{\shuff{5pt}{3.5pt}}{\shuff{3pt}{2.6pt}}{\shuff{3pt}{2.6pt}}}}
\def\shuffle{\, \shuffl \,}

\begin{document}

\maketitle

ABSTRACT. A noncommutative and noncocommutative Hopf algebra on finite topologies $\h_\TT$ is introduced and studied
(freeness, cofreeness, self-duality$\ldots$). Generalizing Stanley's definition of $P$-partitions associated to a special poset, we
define the notion of $T$-partitions associated to a finite topology, and deduce a Hopf algebra morphism from $\h_\TT$ to the Hopf algebra
of packed words $\WQSym$. Generalizing Stanley's decomposition by linear extensions, we deduce a factorization of this morphism,
which induces a combinatorial isomorphism from the shuffle product to the quasi-shuffle product of $\WQSym$.
It is strongly related to a partial order on packed words, here described and studied.\\


AMS CLASSIFICATION. 16T05, 06A11, 54A10.

\tableofcontents

\section*{Introduction}

In his thesis \cite{Stanley}, Stanley introduced the notion of $(P,\omega,m)$-partition associated to a $(P,\omega)$ poset. 
More precisely, a $(\P,\omega)$ poset, or equivalently a special poset, is a finite set $(\P,\leq_\P,\leq)$ with two orders, 
the second being total, see section \ref{section1.2} for examples.
A $(P,\omega,m)$-partition, or, briefly, a $P$-partition, associated to a special poset $\P$ is a map $f:\P\longrightarrow \N$, such that:
\begin{enumerate}
\item If $i\leq_\P j$ in $\P$, then $f(i)\leq f(j)$.
\item If $i\leq_\P j$ and $i>j$ in $\P$, then $f(i)<f(j)$.
\end{enumerate}
Stanley proved \cite{Stanley,Malvenuto} that the set of $P$-partitions of  $\P$ can be decomposed into a disjoint family of subsets
indexed by the set of linear extensions of the partial order $\leq_\P$. \\

Special posets are organized as a Hopf algebra $\h_\SP$, described in \cite{MR2} as a subobject of the Hopf algebra of double posets,
that is to say finite sets with two partial orders. Linear extensions are used to define a Hopf algebra morphism $L$ from $\h_\SP$ 
to the Malvenuto-Retenauer Hopf algebra of permutations $\FQSym$ \cite{MR1,MR3,DHT}. Considering $P$-partitions which are packed words
(which allows to find all $P$-partitions), it is possible to define a Hopf algebra morphism $\Gamma$ from $\h_\SP$ to $\WQSym$, 
the Hopf algebra of packed words. Then Stanley's decomposition allows to define an injective Hopf algebra morphism
$\varphi:\FQSym\longrightarrow \WQSym$, such that the following diagram commutes:
$$\xymatrix{\h_\SP\ar[r]^L \ar[rd]_\Gamma&\FQSym\ar[d]^\varphi\\
&\WQSym}$$\\

Our aim in this the present text is a generalization of Stanley's theorem on $P$-partitions and its applications to combinatorial Hopf algebras.
We here replace special posets by special preposets $(\P,\leq_\P,\leq)$, where $\leq_\P$ is a preorder, that is to say a reflexive  and transitive
relation, and $\leq$ is a total order. By Alexandroff's correspondence, these correspond to topologies on finite sets $[n]=\{1,\ldots,n\}$. 
A construction of a Hopf algebra on finite topologies (up to homeomorphism) is done in \cite{FMP}, where one also can find a brief historic of the subject.
We apply the same construction here and obtain a Hopf algebra $\h_\TT$ on finite topologies, which is noncommutative and noncocommutative.
It is algebraically studied in section \ref{section2}; we prove its freeness and cofreeness (proposition \ref{prop5} and theorem \ref{theo7}), 
show that the Hopf algebra of special posets is both a subalgebra and a quotient of $\h_\TT$ via the construction of a family 
of Hopf algebra morphisms $\theta_q$ (proposition \ref{prop8}). 
A (degenerate) Hopf pairing is also defined on $\h_\TT$, with the help of Zelevinsky's pictures, extending the pairing on special posets of \cite{MR2}. 
The set of topologies on a given set is totally ordered by the refinement; 
using this ordering and a Möbius inversion, we define another basis of  $\h_\TT$, called the \emph{ribbon basis}. 
The product and the coproducts are described in this new basis (theorem \ref{theo12}).\\

The notions of $T$-partitions and linear extensions of a topology are defined in section \ref{section4}. A $T$-partition of a topology $\T$ on the set $[n]$
is introduced in definition \ref{defi13}. Namely, if $\leq_\T$ is the preorder associated to the topology $\T$, a generalized $T$-partition of $\T$
is a surjective map $f:[n]\longrightarrow [p]$ such that: 
\begin{itemize}
\item if $i\leq_\T j$, then $f(i)\leq f(j)$.
\end{itemize}
The $T$-partition $f$ is \emph{strict} if:
\begin{itemize}
\item If $i\leq_\T j$, $i>j$ and not $j\leq_\T i$, then $f(i)<f(j)$.
\item If $i<j<k$, $i\leq_\T k$, $k\leq_\T i$ and $f(i)=f(j)=f(k)$, then $i\leq_\T j$, $j\leq_\T i$, $j\leq_\T k$ and $k\leq_\T i$.
\end{itemize}
The last condition, which is empty for special posets, is necessary to obtain an equivalent of Stanley's decomposition,
as it will be explained later. We now identify any $T$-partition $f$ associated to the topology $\T$ on $[n]$ with the word $f(1)\ldots f(n)$.
A family of Hopf algebra morphisms $\Gamma_{(q_1,q_2,q_3)}$ from $\h_\TT$ to $\WQSym$,
parametrized by triples of scalars, is defined in proposition \ref{prop14}. In particular, for any finite topology $\T$:
\begin{align*}
\Gamma_{(1,1,1)}(\T)&=\sum_{\scriptsize \mbox{$f$ generalized $T$-partition of $\TT$}} f,&
\Gamma_{(1,0,0)}(\T)&=\sum_{\scriptsize \mbox{$f$ strict $T$-partition of $\TT$}} f.
\end{align*}
Linear extensions are introduced in definition \ref{defi15}. They are used to defined a Hopf algebra morphism $L:\h_\TT\longrightarrow \WQSym$, 
up to a change of the product of $\WQSym$: one has to replace its usual product by the shifted shuffling product $\shuffle$, used in \cite{FPT}.
We then look for an equivalent of Stanley's decomposition theorem of $P$-partitions, reformulated in terms of Hopf algebras, that is to say
we look for a Hopf algebra morphism $\varphi_{(q_1,q_2,q_3)}$ making the following diagram commute:
$$\xymatrix{\h_\TT\ar[r]^{L\hspace{15mm}}\ar[rd]_{\Gamma_{(q_1,q_2,q_3)}}&(\WQSym,\shuffle ,\Delta)\ar[d]^{\varphi_{(q_1,q_2,q_3)}}\\
&(\WQSym,.,\Delta)}$$
We prove in proposition \ref{prop21} that such a $\varphi_{(q_1,q_2,q_3)}$ exists if, and only if, $(q_1,q_2,q_3)=(1,0,0)$ or $(0,1,0)$,
which justifies the introduction of strict $T$-partitions. The morphism $\varphi_{(1,0,0)}$ is defined in proposition \ref{prop19}, with the help of a partial order 
on packed words introduced in definition \ref{defi17}; the set decomposition of $T$-partitions is stated in corollary \ref{cor20}. 
Finally, the partial order on packed words is studied in section \ref{section4.5}, with a combinatorial application in corollary \ref{cor26}.\\

The text is organized as follows. The first section recalls the construction of the Hopf algebras $\WQSym$, $\FQSym$ and $\h_\SP$.
The second section deals with the Hopf algebra of topologies and its algebraic study; the ribbon basis is the object of the third section.
The equivalent of Stanley's decomposition, from a combinatorial and a Hopf algebraic point of view, is the object of the last section,
together with the study of the partial order on packed words. \\

{\bf Aknowledgment.} The research leading these results was partially supported by the French National Research Agency under the reference
ANR-12-BS01-0017.\\

{\bf Notations.} \begin{itemize}
\item We work on a commutative base field $\K$, of any characteristic. Any vector space, coalgebra, algebra$\ldots$ of this text is taken over $\K$.
\item For all $n \geq 0$, we put $[n]=\{1,\ldots,n\}$. In particular, $[0]=\emptyset$. We denote by $\N_{>0}$ the set of strictly positive integers.
\end{itemize}

\section{Reminders}

\subsection{$\WQSym$ and $\FQSym$}

Let us first recall the construction of $\WQSym$ \cite{NovelliThibon}. 
\begin{itemize}
\item A packed word is a word $f$ whose letters are in $\N_{>0}$, such that for all $1\leq i\leq j$, 
$$\mbox{$j$ appears in $f$} \Longrightarrow \mbox{$i$ appears in $f$}.$$
Here are the packed words of length $\leq 3$:
$$1=\emptyset;(1); (12),(21),(11);(123),(132),(213),(231),(312),(321),$$
$$(122),(212),(221),(112),(121),(211),(111).$$
\item Let $f=f(1)\ldots f(n)$ be a word whose letters are in $\N_{>0}$. There exists a unique increasing bijection $\phi$
from $\{f(1),\ldots,f(n)\}$ into a set $[m]$. The packed word $Pack(f)$ is $\phi(f(1))\ldots \phi(f(n))$.
\item If $f$ is a word whose letters are in $\N_{>0}$, and $I$ is a subset of $\N_{>0}$, then $f_{\mid I}$ is the subword of $f$ obtained by keeping
only the letters of $f$ which are in $I$.
\end{itemize}

As a vector space, a basis of $\WQSym$ is given by the set of packed words. Its product is defined as follows:
if $f$ and $f'$ are packed words of respective lengths $n$ and $n'$:
$$f.f'=\sum_{\substack{\mbox{\scriptsize $f''$ packed word of length $n+n'$},\\
Pack(f''(1)\ldots f''(n))=f,\\Pack(f''(n+1)\ldots f''(n+n'))=f'}} f''.$$
For example:
\begin{align*}
(112).(12)&=(11212)+(11213)+(11214)+(11223)+(11224)\\
&+(11234)+(11312)+(11323)+(11324)+(11423)\\
&+(22312)+(22313)+(22314)+(22413)+(33412).
\end{align*}

The unit is the empty packed word $1=\emptyset$. 

If $f$ is a packed word, its coproduct in $\WQSym$ is defined by:
$$\Delta(f)=\sum_{k=0}^{\max(f)} f_{\mid [k]} \otimes Pack(f_{\mid \N_{>0}\setminus [k]}).$$
For example: 
\begin{align*}
\Delta((511423))&=1\otimes (511423)+(1)\otimes (4312)+(112)\otimes (321)\\
&+(1123)\otimes (21)+(11423)\otimes (1)+(511423)\otimes 1.
\end{align*}
Then $(\WQSym,.,\Delta)$ is a graded, connected Hopf algebra.\\

We denote by $j$ the involution on packed words defined in the following way: if $f=f(1)\ldots f(n)$ is a packed word
of length $n$, there exists a unique decreasing bijection $\varphi$ from $\{f(1),\ldots,f(n)\}$ into a set $[l]$. We put $j(f)=\varphi(f(1))\ldots \varphi(f(n))$. 
For example, $j((65133421))=(12644356)$. The extension of $j$ to $\WQSym$ is a Hopf algebra isomorphism from $(\WQSym,.,\Delta)$
to $(\WQSym,.,\Delta^{op})$.\\

In particular, permutations are packed words. Note that the subspace of $\WQSym$ generated by all the permutations is a coalgebra, but not a subalgebra:
for example, $(1).(1)=(12)+(21)+(11)$. On the other side, the subspace of $\WQSym$ generated by packed words which are not permutations
is a biideal, and the quotient of $\WQSym$ by this biideal is the Hopf algebra of permutations $\FQSym$ \cite{MR1,DHT}.
As a vector space, a basis of $\FQSym$ is given by the set of all permutations; if $\sigma$ and $\sigma'$ are two permutations
of respective lengths $n$ and $n'$,
$$\sigma.\sigma'=\sum_{\substack{\sigma''\in \S_{n+n'},\\ Pack(\sigma''(1)\ldots \sigma''(n))=\sigma,\\ Pack(\sigma''(n+1)\ldots \sigma''(n+n'))=\sigma'}}
\sigma''=\sum_{\epsilon \in Sh(n,n')} \epsilon \circ (\sigma \otimes \tau) ,$$
where $Sh(n,n')$ is the set of $(n,n')$-shuffles, that is to say permutations $\epsilon \in \S_{n+n'}$ such that $\epsilon(1)<\ldots <\epsilon(n)$
and $\epsilon(n+1)<\ldots<\epsilon(n+n')$. For example:
\begin{align*}
(132).(21)&=(13254)+(14253)+(15243)+(14352)+(15342)\\
&+(15432)+(24351)+(25341)+(25431)+(35421).
\end{align*}
If $\sigma \in \S_n$, its coproduct is given by:
$$\Delta(\sigma)=\sum_{k=0}^{n} \sigma_{\mid [k]} \otimes Pack(\sigma_{\mid \N_{>0}\setminus [k]}).$$
For example: 
\begin{align*}
\Delta((51423))&=1\otimes (51423)+(1)\otimes (4312)+(12)\otimes (321)\\
&+(123)\otimes (21)+(1423)\otimes (1)+(51423)\otimes 1.
\end{align*}
The canonical epimorphism from $\WQSym$ to $\FQSym$ is denoted by $\varpi$.\\

We shall need the standardisation map, which associates a permutation  to any packed word. If $f=f(1)\ldots f(n)$ is a packed word, 
$Std(f)$ is the unique permutation $\sigma \in \S_n$ such that for all $1\leq i,j\leq n$:
\begin{align*}
f(i)<f(j) &\Longrightarrow \sigma(i)<\sigma(j),\\
(f(i)=f(j) \mbox{ and } i<j)&\Longrightarrow \sigma(i)<\sigma(j).
\end{align*}
In particular, if $f$ is a permutation, $Std(f)=f$. Here are examples of standardization of packed words which are not permutations:
\begin{align*}
Std(11)&=(12),&Std(122)&=(123), &Std(212)&=(213),&Std(221)&=(231),\\
Std(112)&=(123),&Std(121)&=(132),&Std(211)&=(312),&Std(111)&=(123).
\end{align*}

\subsection{Special posets}

\label{section1.2} Let us briefly recall the construction of the Hopf algebra on special posets \cite{MR2,Foissy}. A special (double) poset is a family
$(P,\leq,\leq_{tot})$, where $P$ is a finite set, $\leq$ is a partial order on $P$ and $\leq_{tot}$ is a total order on $P$. 
For example, here are the special posets of cardinality $\leq 3$: they are represented by the Hasse graph of $\leq$,
the total order $\leq_{tot}$ is given by the indices of the vertices.
$$1=\emptyset\:; \tdun{$1$}; \tdun{$1$}\tdun{$2$},\tddeux{$1$}{$2$},\tddeux{$2$}{$1$};$$
$$\tdun{$1$}\tdun{$2$}\tdun{$3$},\tddeux{$1$}{$2$}\tdun{$3$},\tddeux{$1$}{$3$}\tdun{$2$},
\tddeux{$2$}{$1$}\tdun{$3$},\tddeux{$2$}{$3$}\tdun{$1$},\tddeux{$3$}{$1$}\tdun{$2$},$$
$$\tddeux{$3$}{$2$}\tdun{$1$}, \tdtroisun{$1$}{$3$}{$2$},\tdtroisun{$2$}{$3$}{$1$},\tdtroisun{$3$}{$2$}{$1$},
\pdtroisun{$1$}{$2$}{$3$},\pdtroisun{$2$}{$1$}{$3$},\pdtroisun{$3$}{$1$}{$2$},
\tdtroisdeux{$1$}{$2$}{$3$},\tdtroisdeux{$1$}{$3$}{$2$},\tdtroisdeux{$2$}{$1$}{$3$},
\tdtroisdeux{$2$}{$3$}{$1$},\tdtroisdeux{$3$}{$1$}{$2$},\tdtroisdeux{$3$}{$2$}{$1$}.$$
If $P=(P,\leq,\leq_{tot})$ and $Q=(Q,\leq,\leq_{tot})$ are two special posets, we define a special posets $P.Q$ in the following way:
\begin{itemize}
\item As a set, $P.Q=P \sqcup Q$.
\item If $i,j \in P$, then $i\leq j$ in $P.Q$ if, and only if, $i\leq j$ in $P$, and $i \leq_{tot} j$ in $P.Q$ if, and only if, $i \leq_{tot} j$ in $P$.
\item If $i,j \in Q$, then $i\leq j$ in $P.Q$ if, and only if, $i\leq j$ in $Q$, and $i \leq_{tot} j$ in $P.Q$ if, and only if, $i \leq_{tot} j$ in $Q$.
\item If $i\in P$ and $j \in Q$, then $i$ and $j$ are not comparable for $\leq$, and $i\leq_{tot} j$.
\end{itemize}
For example, $\tddeux{$1$}{$3$}\tdun{$2$}.\tdtroisun{$1$}{$3$}{$2$}=\tddeux{$1$}{$3$}\tdun{$2$}\tdtroisun{$4$}{$6$}{$5$}$.
The vector space generated by the set of (isoclasses) of special posets is denoted by $\h_\SP$. This product is bilinearly extended
to $\h_\SP$, making it an associative algebra. The unit is the empty special poset $1=\emptyset$. \\

If $P$ is a special poset and $I \subseteq P$, then by restriction $I$ is a special poset. We shall say that $I$ is an ideal of $P$ if for all $i,j \in P$:
$$(i\in I\mbox{ and } i\leq j)\Longrightarrow j\in I.$$
We give $\h_\SP$ the coproduct defined by:
$$\Delta(P)=\sum_{\mbox{\scriptsize $I$ ideal of $P$}} (P\setminus I) \otimes I.$$
For example:
$$\Delta(\tdquatredeux{$2$}{$1$}{$3$}{$4$})=\tdquatredeux{$2$}{$1$}{$3$}{$4$}\otimes 1+1\otimes \tdquatredeux{$2$}{$1$}{$3$}{$4$}
+\tdtroisun{$2$}{$1$}{$3$}\otimes \tdun{$1$}+\tdtroisdeux{$1$}{$2$}{$3$}\otimes \tdun{$1$}+\tddeux{$1$}{$2$}\otimes \tdun{$1$}\tdun{$2$}
+\tddeux{$2$}{$1$} \otimes \tddeux{$1$}{$2$}+\tdun{$1$}\otimes \tddeux{$2$}{$3$}\tdun{$1$}.$$

Let $P=(P,\leq,\leq_{tot})$ be a special poset. A linear extension of $P$ is a total order $\leq'$ extending the partial order $\leq$.
Let $\leq'$ be a linear extension of $P$. Up to a unique isomorphism, we can assume that $P=[n]$ as a totally ordered set. 
For any $i \in [n]$, we denote by $\sigma(i)$ the index of $i$ in the total order $\leq'$. Then $\sigma \in \S_n$, and we now identify $\leq'$ and $\sigma$.
The following map is a surjective Hopf algebra morphism:
$$L:\left\{\begin{array}{rcl}
\h_\SP&\longrightarrow&\FQSym\\
P&\longrightarrow&\displaystyle \sum_{\mbox{\scriptsize $\sigma $ linear extension of $P$}} \sigma.
\end{array}\right.$$ 
For example:
\begin{align*}
L(\pdtroisun{$1$}{$2$}{$3$})&=(312)+(321),&L(\pdtroisun{$2$}{$1$}{$3$})&=(132)+(231),&L(\pdtroisun{$3$}{$1$}{$2$})&=(123)+(213),\\
L(\tdtroisun{$1$}{$3$}{$2$})&=(123)+(132),&L(\tdtroisun{$2$}{$3$}{$1$})&=(213)+(312),&L(\tdtroisun{$3$}{$2$}{$1$})&=(231)+(321),\\
L(\tdtroisdeux{$1$}{$2$}{$3$})&=(123),&L(\tdtroisdeux{$2$}{$1$}{$3$})&=(213),&L(\tdtroisdeux{$3$}{$1$}{$2$})&=(231),\\
L(\tdtroisdeux{$1$}{$3$}{$2$})&=(132),&L(\tdtroisdeux{$2$}{$3$}{$1$})&=(312),&L(\tdtroisdeux{$3$}{$2$}{$1$})&=(321).
\end{align*}

Let $\P$ be a special poset. With the help of the total order of $\P$, we identify $\P$ with the set $[n]$, where $n$ is the cardinality of $\P$.
A $P$-partition of $\P$ is a map $f:\P\longrightarrow [n]$ such that:
\begin{enumerate}
\item If $i\leq_\P j$ in $\P$, then $f(i)\leq f(j)$.
\item If $i\leq_\P j$ and $i>j$ in $\P$, then $f(i)<f(j)$.
\end{enumerate}
We represent a $P$-partition of $\P$ by the word $f(1)\ldots f(n)$. Obviously, if $w=w_1\ldots w_n$ is a word, it is a $P$-partition of the special poset $\P$
if, and only if, $Pack(w)$ is a $P$-partition of $\P$. We define:
$$\Gamma:\left\{\begin{array}{rcl}
\h_\SP&\longrightarrow&\WQSym\\
\P&\longrightarrow&\displaystyle \sum_{\scriptsize \mbox{$w$ packed word, $P$-partition of $\P $}} w.
\end{array}\right.$$
For example:
\begin{align*}
\Gamma(\tdun{$1$})&=(1),\\
\Gamma(\tdun{$1$}\tdun{$2$})&=(12)+(21)+(11),\\
\Gamma(\tddeux{$1$}{$2$})&=(12)+(11),\\
\Gamma(\tddeux{$2$}{$1$})&=(21),\\
\Gamma(\tdtroisun{$1$}{$3$}{$2$})&=(123)+(132)+(122)+(112)+(121)+(111),\\
\Gamma(\tdtroisun{$2$}{$3$}{$1$})&=(213)+(312)+(212)+(211),\\
\Gamma(\tdtroisun{$3$}{$2$}{$1$})&=(231)+(321)+(221).
\end{align*}
We shall prove in section \ref{section4} that $\Gamma$ is a Hopf algebra morphism. \\

{\bf Remark}. There is a natural surjective Hopf algebra morphism $\varrho$ from $\WQSym$ to the Hopf algebra of quasisymmetric functions $\QSym$
 \cite{NovelliThibon2}. For any special poset $\P$:
$$\varrho\circ \Gamma(\P)=\sum_{\scriptsize\mbox{$f$ $P$-partition of $w$}} x_{f(1)}\ldots x_{f(n)}\in \QSym \subseteq \mathbb{Q}[[x_1,x_2,\ldots]].$$
So $\varrho \circ \Gamma(\P)$ is the generating function of $P$ in the sense of \cite{Stanley}.\\

We shall also prove that Stanley's decomposition theorem can be reformulated in the following way: let us consider the map
$$\varphi:\left\{\begin{array}{rcl}
\FQSym&\longrightarrow&\WQSym\\
\sigma&\longrightarrow&\displaystyle \sum_{\scriptsize \mbox{$w$ packed word, $Std(w)=\sigma $}} w.
\end{array}\right.$$
Then $\varphi$ is an injective Hopf algebra morphism, such that $\varphi \circ L=\Gamma$. Combinatorially speaking, for any special poset $\P$:
$$\{\mbox{$P$-partition of $\P $}\}=\bigsqcup_{\scriptsize \mbox{$\sigma $ linear extension of $\P $}}\{w\mid Pack(w)=\sigma\}.$$

\subsection{Infinitesimal bialgebras}

An infinitesimal bialgebra  \cite{LodayRonco} is a triple $(A,m,\Delta)$ such that:
\begin{itemize}
\item $(A,m)$ is a unitary, associative algebra.
\item $(A,\Delta)$ is a counitary, coassociative algebra.
\item For all $x,y \in A$, $\Delta(xy)=(x\otimes 1)\Delta(y)+\Delta(x)(1\otimes y)-x\otimes y$.
\end{itemize}
The standard examples are the tensor algebras $T(V)$, with the concatenation product and the deconcatenation coproduct.
By the rigidity theorem of \cite{LodayRonco}, these are essentially the unique examples:

\begin{theo}\label{theo1}
Let $A$ be a graded, connected, infinitesimal bialgebra. Then $A$ is isomorphic to $T(Prim(A))$ as an infinitesimal bialgebra.
\end{theo}

\section{Topologies on a finite set}

\label{section2}

\subsection{Notations and definitions}

Let $X$ be a set. Recall that a topology on $X$ is a family $\T$ of subsets of $X$, called the open sets of $\T$, such that:
\begin{enumerate}
\item $\emptyset$, $X \in \T$.
\item The union of an arbitrary number of elements of $\T$ is in $\T$.
\item The intersection of a finite number of elements of $\T$ is in $\T$.
\end{enumerate}

Let us recall  from \cite{ErneStege} the bijective correspondence between topologies on a finite set $X$ and preorders on $X$:
\begin{enumerate}
\item Let $\T$ be a topology on the finite set $X$. The relation $\leq_\T$ on $X$ is defined by
$i\leq_\T j$ if any open set of $\T$ which contains $i$ also contains $j$. Then $\leq_\T$ is a preorder, that is to say a reflexive, transitive relation.
Moreover, the open sets of $\T$ are the ideals of $\leq_\T$, that is to say the sets $I\subseteq X$ such that, for all $i,j \in X$:
$$(i\in I\mbox{ and } i\leq_\T j)\Longrightarrow j\in I.$$
\item Conversely, if $\leq$ is a preorder on $X$, the ideals of $\leq$ form a topology on $X$ denoted by $\T_\leq$.
Moreover, $\leq_{T_\leq}=\leq$, and $\T_{\leq_\T}=\T$. Hence, there is a bijection between the set of topologies on $X$ and the 
set of preorders on $X$.
\item Let $\T$ be a topology on $X$. The relation $\sim_\T$ on $X$, defined by $i \sim_\T j$ if $i\leq_\T j$ and $j\leq_\T i$,
is an equivalence on $X$. Moreover, the set $X/\sim_\T$ is partially ordered by the relation defined by
$\overline{i}\leq_\T \overline{j}$ if $i\leq j$. Consequently, we shall represent preorders on $X$ (hence, topologies on $X$)
by the Hasse diagram of $X/\sim_\T$, the vertices being the equivalence classes of $\sim_\T$. 
\end{enumerate}

For example, here are the topologies on $[n]$ for $n\leq 3$:
$$1=\emptyset\:; \tdun{$1$}; \tdun{$1$}\tdun{$2$},\tddeux{$1$}{$2$},\tddeux{$2$}{$1$},\tdun{$1,2$}\hspace{3mm};$$
$$\tdun{$1$}\tdun{$2$}\tdun{$3$},\tddeux{$1$}{$2$}\tdun{$3$},\tddeux{$1$}{$3$}\tdun{$2$},
\tddeux{$2$}{$1$}\tdun{$3$},\tddeux{$2$}{$3$}\tdun{$1$},\tddeux{$3$}{$1$}\tdun{$2$},$$
$$\tddeux{$3$}{$2$}\tdun{$1$}, \tdtroisun{$1$}{$3$}{$2$},\tdtroisun{$2$}{$3$}{$1$},\tdtroisun{$3$}{$2$}{$1$},
\pdtroisun{$1$}{$2$}{$3$},\pdtroisun{$2$}{$1$}{$3$},\pdtroisun{$3$}{$1$}{$2$},
\tdtroisdeux{$1$}{$2$}{$3$},\tdtroisdeux{$1$}{$3$}{$2$},\tdtroisdeux{$2$}{$1$}{$3$},
\tdtroisdeux{$2$}{$3$}{$1$},\tdtroisdeux{$3$}{$1$}{$2$},\tdtroisdeux{$3$}{$2$}{$1$},$$
$$\tdun{$1,2$}\hspace{3mm} \tdun{$3$},\tdun{$1,3$}\hspace{3mm} \tdun{$2$},\tdun{$2,3$}\hspace{3mm} \tdun{$1$},
\tddeux{$1,2$}{$3$}\hspace{3mm},\tddeux{$1,3$}{$2$}\hspace{3mm},\tddeux{$2,3$}{$1$}\hspace{3mm},
\tddeux{$3$}{$1,2$}\hspace{3mm},\tddeux{$2$}{$1,3$}\hspace{3mm},\tddeux{$1$}{$2,3$}\hspace{3mm},
\tdun{$1,2,3$}\hspace{5mm}.$$

The number $t_n$ of topologies on $[n]$ is given by the sequence A000798 in \cite{Sloane}:
$$\begin{array}{|c|c|c|c|c|c|c|c|c|c|c|}
\hline n&1&2&3&4&5&6&7&8&9\\
\hline t_n&1&4&29&355&6\:942&209527&9\:535\:241&642\:779\:354&63\:260\:289\:423\\
\hline \end{array}$$
$$\begin{array}{|c|c|c|c|}
\hline n&10&11&12\\
\hline t_n&8\:977\:053\:873\:043&1\:816\:846\:038\:736\:192&519\:355\:571\:065\:774\:021\\
\hline\end{array}$$

The set of topologies on $[n]$ will be denoted by $\TT_n$, and we put $\displaystyle \TT=\bigsqcup_{n\geq 0} \TT_n$. \\

If $\T$ is a finite topology on a set $X$, then $\iota(\T)=\{X\setminus O\mid O\in \T\}$ is also a finite topology, on the same set $X$.
Consequently, $\iota$ defines a involution of the set $\TT$. The preorder associated to $\iota(\T)$ is $\leq_{\iota(\T)}=\geq_{\T}$.\\

{\bf Notations.} Let $f$ be a packed word of length $n$. We define a preorder $\leq_f$ on $[n]$ by:
$$i\leq_f j \mbox{ if }f(i)\leq f(j).$$
The associated topology is denoted by $\T_f$. The open sets of this topology are the subsets $f^{-1}(\{i,\ldots,\max(f)\})$,
$1\leq i \leq \max(f)$, and $\emptyset$. For example:
$$\T_{(331231)}=\tdtroisdeux{$3,6$}{$4$}{$1,2,5$}\hspace{4mm}.$$

\subsection{Two products on finite topologies}

{\bf Notations.} Let $O\subseteq \N$ and let $n \in \N$. The set $O(+n)$ is the set $\{k+n\mid k\in O\}$.

\begin{defi}
Let $\T\in \TT_n$ and $\T'\in \TT_{n'}$. 
\begin{enumerate}
\item The topology $\T.\T'$ is the topology on $[n+n']$ which open sets are the sets $O \sqcup O'(+n)$, with $O \in \T$ and $O'\in \T'$.
\item The topology $\T \downarrow \T'$ is the topology on $[n+n']$ which open sets are the sets $O \sqcup [n'](+n)$, with $O \in \T$,
and $O'(+n)$, with $O'\in \T'$. 
\end{enumerate}\end{defi}

{\bf Example.} $\tdtroisun{$1$}{$3$}{$2$}.\tddeux{$2$}{$1$}=\tdtroisun{$1$}{$3$}{$2$}\tddeux{$5$}{$4$}$ and
$\tdtroisun{$1$}{$3$}{$2$}\downarrow \tddeux{$2$}{$1$}=\pdcinq{$1$}{$3$}{$2$}{$5$}{$4$}$.

\begin{prop} 
These two products are associative, with $\emptyset=1$ as a common unit.
\end{prop}

\begin{proof} Obviously, for any $\T\in \TT$, $1.\T=\T.1=1\downarrow \T=\T\downarrow 1=\T$.
Let $\T\in \TT_n$, and $\T'\in \TT_{n'}$. The preorder associated to $\T.\T'$ is:
$$\{(i,j)\mid i\leq_\T j\}\sqcup \{(i+n,j+n)\mid i\leq_{\T'} j\}.$$
The preorder associated to $\T\downarrow \T'$ is:
$$\{(i,j)\mid i\leq_\T j\}\sqcup \{(i+n,j+n)\mid i\leq_{\T'} j\}\sqcup \{(i,j)\mid 1\leq i\leq n<j\leq n+n'\}.$$
Let $\T\in \TT_n$, $\T'\in \TT_{n'}$ and $\T''\in \TT_{n''}$.
The preorders associated to $(\T.\T').\T''$ and to $\T.(\T'.\T'')$ are both equal to:
$$\{(i,j)\mid i\leq_\T j\}\sqcup \{(i+n,j+n)\mid i\leq_{\T'} j\}\sqcup \{(i+n+n',j+n+n')\mid i\leq_{\T''}j\}.$$
So $(\T.\T').\T''=\T.(\T'.\T'')$. The preorders associated to $(\T\downarrow \T')\downarrow \T''$ 
and to $\T\downarrow (\T'\downarrow \T'')$ are both equal to:
$$\{(i,j)\mid i\leq_\T j\}\sqcup \{(i+n,j+n)\mid i\leq_{\T'} j\}\sqcup \{(i+n+n',j+n+n')\mid i\leq_{\T''}j\}$$
$$\sqcup\{(i,j)\mid 1\leq i\leq n < j\leq n+n'+n''\}\sqcup\{(i,j)\mid n<i\leq n+n'<j\leq n+n'+n''\}.$$
So $(\T\downarrow \T')\downarrow \T''=\T\downarrow (\T'\downarrow \T'')$. \end{proof}

\begin{defi}\begin{enumerate}
\item We denote by $\h_\TT$ the vector space generated by $\TT$. It is graded, the elements of $\TT_n$ being homogeneous of degree $n$.
We extend the two products defined earlier on $\h_\TT$.
\item Let $\T \in \TT$, different from $1$. 
\begin{enumerate}
\item We shall say that $\T$ is indecomposable if it cannot be written as $\T=\T'.\T''$, with $\T',\T''\neq 1$.
\item We shall say that $\T$ is $\downarrow$-indecomposable if it cannot be written as $\T=\T'\downarrow \T''$, with $\T',\T''\neq 1$.
\item We shall say that $\T$ is bi-indecomposable if it is both indecomposable and $\downarrow$-indecomposable.
\end{enumerate} \end{enumerate}\end{defi}

Note that $(\h_\T,.,\downarrow)$ is a 2-associative algebra \cite{LodayRonco}, that is to say an algebra with two associative products sharing the same unit.

\begin{prop}\label{prop5}\begin{enumerate}
\item The associative algebra $(\h_\TT,.)$ is freely generated by the set of indecomposable topologies.
\item The associative algebra $(\h_\TT,\downarrow)$ is freely generated by the set of $\downarrow$-indecompo\-sable topologies.
\item The 2-associative algebra $(\h_\TT,.,\downarrow)$ is freely generated by the set of bi-indecomposable topologies.
\end{enumerate}\end{prop}

\begin{proof} 1.  An easy induction on the degree proves that any $\T\in\TT$ can be written as $\T=\T_1.\ldots.\T_k$,
with $\T_1,\ldots,\T_k$ indecomposable.

 Let us assume that $\T=\T_1.\cdots.\T_k=\T'_1.\cdots.\T'_l$, with $\T_1,\ldots,\T_k,\T'_1,\ldots,\T'_l$ indecomposable topologies. 
Let $m$ be the smallest integer $\geq 1$ such that for all $1\leq i \leq m<j\leq n$, $i$ and $j$ are not comparable for $\leq_\T$. 
By definition of the product $.$, for all $i\leq deg(\T_1)$, for all $j>deg(\T_1)$, $i$ and $j$ are not comparable for $\leq_\T$, 
so $m\leq deg(\T_1)$. Let $\T'$ be the restriction of the topology $\T_1$ to $\{1,\ldots,m\}$ and $\T''$ be the restriction of the topology $\T_1$ 
to $\{m+1,\ldots,deg(\T_1)\}$, reindexed to $\{1,\ldots,deg(\T_1)-m\}$. By definition of $m$, $\T_1=\T'.\T''$.
As $\T_1$ is indecomposable, $\T'=1$ or $\T''=1$; as $m\geq 1$, $\T''=1$, so $\T'=\T_1$. Similarly, $\T'=\T'_1=\T_1$.
The restriction of $\T$ to $\{m+1,\ldots,deg(\T)\}$, after a reindexation, gives $\T_2.\cdots.\T_k=\T'_2.\cdots.\T'_l$. 
We conclude by an induction on the degree of $\T$. \\

2. Similar proof. For the unicity of the decomposition, use the smallest integer $m\geq 1$ such that for all $i\leq m<j\leq n$, $i\leq_\T j$. \\

3. {\it First step.} Let $\T\in \TT_n$, $n\geq 1$. Let us assume that $\T$ is not $\downarrow$-indecomposable.
Then $\T=\T'\downarrow \T''$, with $\T',\T''\neq 1$, so $1\leq_\T n$: this implies that $\T$ is indecomposable.
 Hence, one, and only one, of the following assertions holds:
\begin{itemize}
\item $\T$ is indecomposable and not $\downarrow$-indecomposable.
\item $\T$ is not indecomposable and $\downarrow$-indecomposable.
\item $\T$ is bi-indecomposable.
\end{itemize}

{\it Second step.} Let $(A,.\downarrow)$ be a 2-associative algebra, and let $a_\T\in A$ for any bi-indecomposable $\T\in \TT$.
Let us prove that there exists a unique morphism of 2-associative algebras $\phi:\h_\TT\longrightarrow A$, such that $\phi(\T)=a_\T$
for all bi-indecomposable $\T\in \TT$. The proof will follow, since $\h_\TT$ satisfies the universal property of the free $2$-associative 	algebra
generated by the bi-indecomposable elements. We define $\phi(\T)$ for $\T \in \TT$ by induction on $deg(\T)$ in the following way:
\begin{enumerate}
\item $\phi(1)=1_A$.
\item If $\T$ is bi-indecomposable, then $\phi(\T)=a_\T$.
\item If $\T$ is indecomposable and not $\downarrow$-indecomposable, we write uniquely $\T=\T_1\downarrow \ldots \downarrow \T_k$,
with $k\geq 2$, $\T_1,\ldots,\T_k \in \TT$, $\downarrow$-indecomposable. Then $\phi(\T)=\phi(\T_1)\downarrow\ldots
\downarrow \phi(\T_k)$.
\item If $\T$ is not indecomposable and $\downarrow$-indecomposable, we write uniquely $\T=\T_1.\cdots.\T_k$,
with $k\geq 2$, $\T_1,\ldots,\T_k \in \TT$, indecomposable. Then $\phi(\T)=\phi(\T_1). \cdots.\phi(\T_k)$.
\end{enumerate}
By the first step, $\phi$ is well-defined. By the unicity of the decomposition into decomposables or $\downarrow$-indecomposables,
$\phi$ is a morphism of 2-associative algebras. \end{proof}

We denote by $F(X)$ the generating formal series of all topologies on $[n]$, by $F_I(X)$ the formal series of indecomposable
topologies on $[n]$, by $F_{\downarrow I}(X)$ the formal series of $\downarrow$-indecomposable topologies on $[n]$, and by
$F_{BI}(X)$ the formal series on bi-indecomposable topologies on $[n]$. Then:
$$F_I(X)=F_{\downarrow I}(X)=\frac{F(X)-1}{F(X)},\:
F_{BI}(X)=\frac{-2+3F(X)-F(X)^2}{F(X)}.$$
This gives:
$$\begin{array}{|c|c|c|c|c|c|c|c|c|c|}
\hline n&1&2&3&4&5&6&7&8&9\\
\hline I,\downarrow I&1&3&22&292&6\:120&193\:594&9\:070\:536&622\:336\:756&61\:915\:861\:962\\
\hline BI&1&2&15&229&5\:298&177\:661&8\:605\:831&601\:894\:158&60\:571\:434\:501\\
\hline \end{array}$$
$$\begin{array}{|c|c|c|c|c|}
\hline n&10&11&12\\
\hline I,\downarrow I&8\:846\:814\:822\:932&
1\:798\:543\:906\:246\:948&515\:674\:104\:905\:890\:202\\
\hline BI&8\:716\:575\:772\:821&1\:780\:241\:773\:757\:704&511\:992\:638\:746\:006\:383\\
\hline \end{array}$$

\subsection{A coproduct on finite topologies}

{\bf Notations}.  \begin{enumerate} 
\item Let $X$ be a finite, totally ordered set of cardinality $n$, and $\T$ a topology on $X$. 
There exists a unique increasing bijection $\phi$ from $X$ to $[n]$. We denote by $Std(\T)$ the topology on $[n]$ defined by:
$$Std(\T)=\{\phi(O)\mid O\in \T\}.$$
It is an element of $\TT_ n$.
\item Let $X$ be a finite set, and $\T$ be a topology on $X$. For any $Y\subseteq X$, we denote by $\T_{\mid Y}$ the topology induced
by $\T$ on $Y$, that is to say:
$$\T_{\mid Y}=\{O\cap Y\mid O\in \T\}.$$
Note that if $Y$ is an open set of $\T$, $\T_{\mid Y}=\{O\in \T\mid O\subseteq Y\}$.
\end{enumerate}

\begin{prop}
Let $\T\in \TT_n$, $n \geq 1$. We put:
$$\Delta(\T)=\sum_{O\in \T}Std(\T_{\mid [n] \setminus O})\otimes Std(\T_{\mid O}).$$
Then:
\begin{enumerate}
\item $(\h_\TT,.,\Delta)$ is a graded Hopf algebra.
\item $(\h_\TT,\downarrow,\Delta)$ is a graded infinitesimal bialgebra.
\item The involution $\iota$ defines a Hopf algebra isomorphism from $(\h_\TT,.,\Delta)$ to $(\h_\TT,.,\Delta^{op})$.
\end{enumerate}\end{prop}

\begin{proof} Let $\T\in \TT_n$, $n\geq 0$. Then:
\begin{align*}
&(\Delta \otimes Id)\circ \Delta(\T)\\
&=\sum_{O\in \T,\: O'\in \T_{\mid [n]\setminus O} } Std((\T_{\mid [n]\setminus O})_{\mid ([n]\setminus O)\setminus O'}) \otimes 
Std((\T_{\mid [n]\setminus O})_{\mid O'}) \otimes Std(\T_{\mid O})\\
&=\sum_{O\in \T,\: O'\in \T_{\mid [n]\setminus O} } 
Std(\T_{\mid [n]\setminus (O \sqcup O')}) \otimes Std(\T_{\mid O'}) \otimes Std(\T_{\mid O}).
\end{align*}
If $O \in \T$ and $O' \in \T_{\mid [n] \setminus O}$, then $O \sqcup O'$ is an open set of $\T$.
Conversely, if $O_1\subseteq O_2$ are open sets of $\T$, then $O_2\setminus O_1 \in \T_{\mid [n]\setminus O_1}$.
Putting $O_1=O$ and $O_2=O\sqcup O'$:
$$(\Delta \otimes Id)\circ \Delta(\T)=\sum_{O_1\subseteq O_2 \in \T} Std(\T_{\mid [n]\setminus O_2})\otimes 
Std(\T_{\mid O_2 \setminus O_1})\otimes Std(\T_{\mid O_1}).$$
Moreover:
\begin{align*}
(Id \otimes \Delta)\circ \Delta(\T)&=\sum_{O\in \T,\: O' \in \T_{\mid O}}
Std(\T_{\mid [n]\setminus O})\otimes Std(\T_{\mid O\setminus O'}) \otimes Std(\T_{\mid O'})
\end{align*}
If $O$ is an open set of $\T$ and $O'$ is an open set of $\T_{\mid O}$, then $O'$ is an open set of $\T$. Hence,
putting $O_1=O'$ and $O_2=O$:
$$(Id \otimes \Delta)\circ \Delta(\T)=\sum_{O_1\subseteq O_2 \in \T} Std(\T_{\mid [n]\setminus O_2})\otimes 
Std(\T_{\mid O_2 \setminus O_1})\otimes Std(\T_{\mid O_1}).$$
This proves that $\Delta$ is coassociative. It is obviously homogeneous of degree $0$. Moreover, $\Delta(1)=1\otimes 1$
and for any $\T\in \TT_n$, $n\geq 1$:
$$\Delta(\T)=\T \otimes 1+1\otimes \T+\sum_{\emptyset \subsetneq O\subsetneq [n]}
Std(\T_{\mid [n] \setminus O})\otimes Std(\T_{\mid O}).$$
So $\Delta$ has a counit.

Let $\T\in \TT_n$, $\T' \in \TT_{n'}$, $n,n'\geq 0$. By definition of $\T.\T'$:
\begin{align*}
\Delta(\T.\T')&=\sum_{O\in \T,O'\in \T'} Std((\T.\T')_{\mid [n+n']\setminus O.O'})
\otimes Std((\T.\T')_{\mid O.O'})\\
&=\sum_{O\in \T,O'\in \T'} Std(\T_{\mid [n]\setminus O}).Std(\T'_{[n']\setminus O'})
\otimes Std(\T_{\mid O}).Std(\T_{\mid O'})\\
&=\sum_{O\in \T,O'\in \T'} \left(Std(\T_{\mid [n]\setminus O})
\otimes Std(\T_{\mid O})\right).\left(Std(\T'_{\mid [n']\setminus O'}) \otimes Std(\T_{\mid O'})\right)\\
&=\Delta(\T).\Delta(\T').
\end{align*}
Hence, $(\h_\T,.,\Delta)$ is a Hopf algebra. \\

By definition of $\T \downarrow \T'$:
\begin{align*}
\Delta(\T \downarrow \T')&=\sum_{O\in \T, O\neq \emptyset}
Std\left((\T\downarrow \T')_{\mid [n+n']\setminus (O\downarrow [n'])}\right)
\otimes Std\left((\T\downarrow \T')_{\mid O\downarrow [n']}\right)\\
&+\sum_{O'\in \T',O'\neq [n']}
Std\left((\T \downarrow \T')_{\mid [n+n']\setminus O'(+n)}\right)
\otimes Std\left((\T\downarrow \T')_{\mid O'(+n)}\right)\\
&+Std\left((\T\downarrow \T')_{\mid [n+n']\setminus [n'](+n)}\right)\otimes
Std\left((\T\downarrow \T')_{ [n'](+n)}\right)\\
&=\sum_{O\in \T, O\neq \emptyset}Std(\T_{\mid  [n]\setminus O})\otimes
Std(\T_{\mid O})\downarrow \T' \\
&+\sum_{O'\in \T',O'\neq [n']}\T\downarrow Std(\T'_{\mid [n']\setminus O'})
\otimes Std(\T'_{\mid O'})+\T \otimes \T'\\
&=\sum_{O\in \T, O\neq \emptyset}\left(Std(\T_{\mid  [n]\setminus O})\otimes
Std(\T_{\mid O})\right)\downarrow (1\otimes \T') \\
&+\sum_{O'\in \T',O'\neq [n']}(\T \otimes 1)\downarrow \left(Std(\T'_{\mid [n']\setminus O'})
\otimes Std(\T'_{\mid O'})\right)+\T \otimes \T'\\
&=(\Delta(\T)-\T\otimes 1)\downarrow (1\otimes \T')+(\T \otimes 1)\downarrow (\Delta(\T)-1\otimes \T')+\T\otimes \T'\\
&=\Delta(\T)\downarrow (1\otimes \T')+(\T \otimes 1)\downarrow \Delta(\T)-\T\otimes \T'.
\end{align*}
Hence, $(\h_\TT,\downarrow,\Delta)$ is an infinitesimal bialgebra. \\

For all $\T,\T' \in \TT$, $\iota(\T.\T')=\iota(\T).\iota(\T')$. Moreover:
\begin{align*}
\Delta(\iota(\T))&=\sum_{O\in \T}Std(\iota(\T)_{\mid O}) \otimes Std(\iota(\T)_{\mid [n] \setminus O})\\
&=\sum_{O\in \T}\iota(Std(\T_{\mid O})) \otimes \iota(Std(\T_{\mid [n] \setminus O}))\\
&=(\iota \otimes \iota)\circ \Delta^{op}(\T).
\end{align*}
So $\iota$ is a Hopf algebra morphism from $\h_\TT$ to $\h_\TT^{cop}$. \end{proof}

As a consequence of theorem \ref{theo1}:

\begin{theo} \label{theo7}
The graded, connected coalgebra $(\h_\TT,\Delta)$ is cofree, that is to say is isomorphic to the tensor algebra on the space of its primitive elements
with the deconcatenation coproduct.
\end{theo}

{\bf Remark.} Forgetting the total order on $[n]$, that is to say considering isoclasses of finite topologies, we obtain the Hopf algebra of finite spaces
of \cite{FMP} as a quotient of $\h_\TT$; the product $\downarrow$ induces the product $\succ$ on finite spaces.

\subsection{Link with special posets}

Let $\T\in \TT$. We put:
$$c(\T)=deg(\T)-\sharp\{\mbox{equivalence classes of}\sim_\T\}.$$
Note that $c(\T)\geq 0$. Moreover, $c(\T)=0$ if, and only if, the relation $\sim_\T$ is the equality, or equivalently
if the preorder $\leq_\T$ is an order, that is to say if $\T$ is $T_0$ \cite{Stong}.\\

If $\T\in \TT_n$, $n\geq 0$, is $T_0$, then for any open set $O$ of $\T$,  $\T_{\mid O}$ and $\T_{\mid [n] \setminus O}$ are also $T_0$. 
Moreover, if $\T$ and $\T'$ are $T_0$, then $\T.\T'$ and $\T'\downarrow \T'$ also are. Hence, the subspace $\h_{\TT_0}$ of $\h_\TT$ generated by 
$T_0$ topologies is a Hopf subalgebra. Considering $T_0$ topologies as special posets, it is isomorphic to the Hopf algebra of special posets $\h_\SP$:
this defines an injective Hopf algebra morphism from $\h_\SP$ to $\h_{\TT}$.
We now identify $\h_\SP$ with its image by this morphism, that is to say with $\h_{\TT_0}$. 
Note that $\h_\SP$ is stable under $\downarrow$, so is a Hopf $2$-associative subalgebra of $\h_{\TT}$.\\

{\bf Notation.} Let $\T\in \TT_n$, $n\geq 0$. We denote by $\overline{\T}$ the special poset $Std([n]/\sim_\T)$,
resulting on the set of equivalence classes of $\sim_\T$, where the elements of $[n]/\sim_\T$, that is to say the equivalence classes of $\sim_\T$, 
are totally ordered by the smallest element of each class. In this way, $\overline{\T}$ is a special poset. \\

{\bf Examples.}
\begin{align*}
\overline{\tdun{$1,2$}\hspace{3mm} \tdun{$3$}}&=\tdun{$1$}\tdun{$2$},&\overline{\tddeux{$1,2$}{$3$}\hspace{3mm}}&=\tddeux{$1$}{$2$},&
\overline{\tddeux{$3$}{$1,2$}\hspace{3mm}}&=\tddeux{$2$}{$1$},&\overline{\tdun{$1,2$}\hspace{3mm}}&=\tdun{$1$},\\
\overline{\tdun{$1,3$}\hspace{3mm} \tdun{$2$}}&=\tdun{$1$}\tdun{$2$},&\overline{\tddeux{$1,3$}{$2$}\hspace{3mm}}&=\tddeux{$1$}{$2$},&
\overline{\tddeux{$2$}{$1,3$}\hspace{3mm}}&=\tddeux{$2$}{$1$},&\overline{\tdun{$1,2,3$}\hspace{5mm}}&=\tdun{$1$}.\\
\overline{\tdun{$2,3$}\hspace{3mm} \tdun{$1$}}&=\tdun{$1$}\tdun{$2$},&\overline{\tddeux{$2,3$}{$1$}\hspace{3mm}}&=\tddeux{$2$}{$1$},&
\overline{\tddeux{$1$}{$2,3$}\hspace{3mm}}&=\tddeux{$1$}{$2$};
\end{align*}

\begin{prop} \label{prop8}
 Let $q\in \K$. The following map is a surjective morphism of Hopf 2-associative algebras:
$$\theta_q:\left\{\begin{array}{rcl}
\h_\TT&\longrightarrow&\h_\SP\\
\T&\longrightarrow&q^{c(\T)}\overline{\T}.
\end{array}\right.$$
\end{prop}

\begin{proof} If $\T,\T'\in \TT$, then $\overline{\T.\T'}=\overline{\T}.\overline{\T'}$ and $\overline{\T\downarrow \T'}=\overline{\T}
\downarrow \overline{\T'}$. Moreover, $deg(\T.\T')=deg(\T\downarrow \T')=deg(\T)+deg(\T')$, and the number of equivalence
classes of $\sim_{\T.\T'}$ and $\sim_{\T\downarrow \T'}$ are both equal to the sum of the number of equivalence classes
of $\sim_\T$ and $\sim_\T'$. Hence, $c(\T.\T')=c(\T\downarrow \T')=c(\T)+c(\T')$, and:
$$\theta_q(\T.\T')=q^{c(\T.\T')}\overline{\T.\T'}=q^{c(\T)}q^{c(\T')}\overline{\T}.\overline{\T'}=\theta_q(\T).\theta_q(\T'),$$
$$\theta_q(\T\downarrow\T')=q^{c(\T\downarrow\T')}\overline{\T\downarrow\T'}
=q^{c(\T)}q^{c(\T')}\overline{\T}\downarrow\overline{\T'}=\theta_q(\T)\downarrow\theta_q(\T'),$$
so $\theta_q$ is a 2-associative algebra morphism.
If $\T\in \TT_n$, $n\geq 1$, then any open set of $\T$ is a union of equivalence classes of $\sim_\T$. So there is a bijection:
$$\left\{\begin{array}{rcl}
\{\mbox{open sets of $\T $}\}&\longrightarrow&\{\mbox{ideals of $\overline{\T}$}\}\\
O&\longrightarrow&Std(O/\sim_\T)
\end{array}\right.$$
Moreover, $c(\T)=c(\T_{\mid [n] \setminus O})+c(\T_{\mid O})
=c(Std(\T_{\mid [n] \setminus O}))+c(Std(\T_{\mid O}))$.
If $\T$ has $k$ equivalence classes, we obtain:
\begin{align*}
\Delta\circ \theta_q(\T)&=q^{c(\T)}\Delta(\overline{\T})\\
&=q^{c(\T)}\sum_{O\in \T}(([n] \setminus O)/\sim_\T)\otimes (O/\sim_\T)\\
&=\sum_{O\in \T}
q^{c(Std(\T_{\mid [n] \setminus O}))}q^{c(Std(\T_{\mid O}))}
\overline{Std(\T_{\mid [n] \setminus O})}\otimes \overline{Std(\T_{\mid O})}\\
&=(\theta_q \otimes \theta_q)\circ \Delta(\T).
\end{align*}
If $\T \in \TT$ is $T_0$, then $\overline{\T}=\T$ and $c(\T)=0$, so $\theta_q(\T)=\T$: $\theta_q$ is surjective. \end{proof} 

We obtain a commutative diagram of Hopf 2-associative algebras:
$$\xymatrix{&\h_\TT\ar@{->>}^{\theta_q}[d]\\
\h_\SP\ar@{^{(}->}[ru]^{\iota}\ar[r]_{Id}&\h_\SP}$$

{\bf Examples.}
\begin{align*}
\theta_q(\tdun{$1,2$}\hspace{3mm} \tdun{$3$})&=q\tdun{$1$}\tdun{$2$},&\theta_q(\tddeux{$1,2$}{$3$}\hspace{3mm})&=q\tddeux{$1$}{$2$},&
\theta_q(\tddeux{$3$}{$1,2$}\hspace{3mm})&=q\tddeux{$2$}{$1$},&\theta_q(\tdun{$1,2$}\hspace{3mm})&=q\tdun{$1$},\\
\theta_q(\tdun{$1,3$}\hspace{3mm} \tdun{$2$})&=q\tdun{$1$}\tdun{$2$},&\theta_q(\tddeux{$1,3$}{$2$}\hspace{3mm})&=q\tddeux{$1$}{$2$},&
\theta_q(\tddeux{$2$}{$1,3$}\hspace{3mm})&=q\tddeux{$2$}{$1$},&\theta_q(\tdun{$1,2,3$}\hspace{5mm})&=q^2\tdun{$1$}.\\
\theta_q(\tdun{$2,3$}\hspace{3mm} \tdun{$1$})&=q\tdun{$1$}\tdun{$2$},&\theta_q(\tddeux{$2,3$}{$1$}\hspace{3mm})&=q\tddeux{$2$}{$1$},&
\theta_q(\tddeux{$1$}{$2,3$}\hspace{3mm})&=q\tddeux{$1$}{$2$}.
\end{align*}

{\bf Remarks.} \begin{enumerate}
\item In particular, for any $\T\in \TT$:
$$\theta_0(\T)=\left\{\begin{array}{l}
\T\mbox{ if $\T$ is $T_0$},\\
0\mbox{ otherwise.}
\end{array}\right.$$
\item $\theta_q$ is homogeneous for the gradation of $\h_\TT$ by the cardinality if, and only if, $q=0$. It is always homogeneous
for the gradation by the number of equivalence classes (note that this gradation is not finite-dimensional).
\end{enumerate}

\subsection{Pictures and duality}

The concept of pictures between tableaux was introduced by Zelevinsky in \cite{Zelevinsky}, and generalized to pictures between double posets 
by Malvenuto and Reutenauer in \cite{MR2}. We now generalize this for finite topologies, to obtain a Hopf pairing on $\h_\TT$. \\

{\bf Notations.}  Let $\leq_\T$ be a preorder on $[n]$, and let $i,j \in [n]$. We shall write $i<_\T j$ if ($i\leq_\T j$ and not $j\leq_\T i$). 

\begin{defi}
Let $\T\in \TT_k$, $\T\in \TT_l$, and let $f:[k]\longrightarrow [l]$. We shall say that $f$ is a picture from $\T$ to $\T'$ if:
\begin{itemize}
\item $f$ is bijective;
\item For all $i,j \in [k]$, $i<_\T j$ $\Longrightarrow$ $f(i)< f(j)$;
\item For all $i,j \in [k]$, $f(i)<_{\T'} f(j)$ $\Longrightarrow$ $i<j$;
\item For all $i,j \in [k]$, $i\sim_\T j$ $\Longleftrightarrow$ $f(i) \sim_{\T'} f(j)$.
\end{itemize}
The set of pictures between $\T$ and $\T'$ is denoted by $Pic(\T,\T')$.
\end{defi}

\begin{prop}
We define a pairing on $\h_\TT$ by $\langle \T,\T'\rangle=\sharp Pic(\T,\T')$ for all $\T,\T'\in \TT$. Then this pairing, when extended by linearity,
is a symmetric Hopf pairing. Moreover, $\iota$ is an isometry for this pairing.
\end{prop}

\begin{proof} First, if $\T,\T'\in \TT$, $Pic(\T',\T)=\{f^{-1}\mid f\in Pic(\T,\T')\}$, so $\langle \T',\T\rangle=\langle \T,\T'\rangle$: the pairing is symmetric. \\

We fix $\T_1,\T_2,\T\in \TT$, of respective degrees $n_1$, $n_2$ and $n$. Let $f\in Pic(\T_1.\T_2,\T)$.
Let $x \in [n_1+n_2]\setminus [n_1]$ and $y\in [n_1+n_2]$, such that $f(x)\leq_\T f(y)$. Two cases can occur:
\begin{itemize}
\item $f(x)<_\T f(y)$. Then $x<y$, so $y\in [n_1+n_2]\setminus [n_1]$.
\item $f(x) \sim_\T f(y)$. Then $x \sim_{\T_1.\T_2} y$. By definition of $\T_1.\T_2$, $[n_1+n_2]\setminus [n_1]$
is an open set of $\T_1.\T_2$, so $y \in [n_1+n_2]\setminus [n_1]$. 
\end{itemize}
In both cases, $y \in [n_1+n_2]\setminus [n_1]$, so $f([n_1+n_2]\setminus [n_1])$ is an open set of $\T$, which we denote by $O_f$.
Moreover, by restriction, $Std(f_{\mid [n_1]})$ is a picture between $\T_1$ and $Std(\T_{\mid [n]\setminus O_f})$
and $Std(f_{\mid [n_1+n_2]\setminus [n_1]})$ is a picture between $\T_2$ and $Std(\T_{\mid O})$. We can define a map:
$$\phi:\left\{\begin{array}{rcl}
Pic(\T_1.\T_2,\T)&\longrightarrow&\displaystyle 
\bigsqcup_{O\in \T} Pic(\T_1,Std(\T_{\mid [n]\setminus O}))\times Pic(\T_2,Std(\T_{\mid O}))\\
f&\longrightarrow&(Std(f_{\mid [n_1]}),Std(f_{\mid [n_1+n_2]\setminus [n_1]}))
\end{array}\right.$$
This map is clearly injective. Let $O\in \T$, $(f_1,f_2) \in Pic(\T_1,Std(\T_{\mid [n]\setminus O}))\times Pic(\T_2,Std(\T_{\mid O}))$.
Let $f$ be the unique bijection $[n_1+n_2]\longrightarrow [n]$ such that  $f([n_1+n_2]\setminus [n_1])=O$, 
$f_1=Std(f_{\mid [n_1]})$ and $f_2=Std(f_{\mid [n_1+n_2]\setminus [n_1]})$. Let us prove that $f$ is a picture between $\T_1.\T_2$ and $\T$. 
\begin{itemize}
\item If $i<_{\T_1.\T_2} j$ in $[n_1+n_2]$, then $(i,j) \in [n_1]^2$ or $(i,j)\in ([n_1+n_2]\setminus [n_1])^2$.
In the first case, $f_1(i)<f_1(j)$, so $f(i)<f(j)$. In the second case, $f_2(i-n_1)<f_2(j-n_1)$, so $f(i)<f(j)$.
\item If $f(i)<_\T f(j)$ in $\T$, as $f([n_1+n_2]\setminus [n_1])=O$ is an open set, $(i,j) \in [n_1]^2$ or $(i,j) \in ([n_1+n_2]\setminus [n_1])^2$
or $(i,j) \in [n_1]\times ([n_1+n_2]\setminus [n_1])$. In the first case, $f_1(i)<_{Std(\T_{\mid [n]\setminus O})} f_1(j)$, so $i<j$. In the second case, 
$f_2(i-n_1)<_{Std(\T_{\mid O})} f_2(j-n_1)$, so $i-n_1<j-n_1$ and $i<j$. In the last case, $i<j$.
\item If $i\sim_{\T_1.\T_2} j$, by definition of $\T_1.\T_2$, $(i,j)\in [n_1]^2$ or $(i,j)\in ([n_1+n_2]\setminus [n_1])^2$.
In the first case, $f_1(i)\sim_{Std(\T_{\mid [n]\setminus O})} f_1(j)$, so $f(i) \sim_\T f(j)$.
In the second case, $f_2(i-n_1)\sim_{Std(\T_{\mid O})} f_2(j-n_1)$, so $f(i)\sim_\T f(j)$.
\item If $f(i) \sim_\T f(j)$, as $O=f([n_1+n_2]\setminus [n_1])$ is an open set of $\T$, both $O$ and $f([n_1])=[n]\setminus O$ are 
stable under $\sim_\T$, so $(i,j) \in [n_1]^2$ or $(i,j)\in ([n_1+n_2]\setminus [n_1])^2$. In the first case, $f_1(i) \sim_{Std(\T_{\mid [n]\setminus O})} f_1(j)$,
so $i\sim_{\T_1} j$ and $i\sim_{\T_1.\T_2} j$. In the second case, $f_2(i-n_1) \sim_{Std(\T_{\mid O})} f_2(j-n_1)$, so $i-n_1 \sim_{\T_2} j-n_1$ and
$i\sim_{\T_1.\T_2} j$.
\end{itemize}
Finally, $\phi$ is bijective. We obtain:
\begin{align*}
\langle \T_1.\T_2,\T\rangle&=\sharp Pic(\T_1.\T_2,\T)\\
&=\sum_{O\in \T} \sharp Pic(\T_1,Std(\T_{\mid [n]\setminus O}))\times \sharp Pic(\T_2,Std(\T_{\mid O}))\\
&=\sum_{O\in \T} \langle \T_1,Std(\T_{\mid [n]\setminus O})\rangle \langle Pic(\T_2,Std(\T_{\mid O})\rangle\\
&=\langle \T_1\otimes \T_2,\Delta(\T)\rangle.
\end{align*}
So $\langle-,-\rangle$ is a Hopf pairing. \\

Let $\T,\T' \in \TT$. It is not difficult to show that $Pic(\iota(\T),\iota(\T'))=j(Pic(\T,\T'))$, so $\langle \iota(\T),\iota(\T')\rangle=\langle \T,\T'\rangle$. \end{proof}

{\bf Remark.} Here is the matrix of the pairing in degree $2$:
$$\begin{array}{|c|c|c|c|c|}
\hline&\tdun{$1$}\tdun{$2$}&\tddeux{$1$}{$2$}&\tddeux{$2$}{$1$}&\tdun{$1,2$}\hspace{2mm}\\
\hline \tdun{$1$}\tdun{$2$}&2&1&1&0\\
\hline \tddeux{$1$}{$2$}&1&1&0&0\\
\hline \tddeux{$2$}{$1$}&1&0&1&0\\
\hline \tdun{$1,2$}\hspace{2mm}&0&0&0&2\\
\hline \end{array}$$
So this pairing is degenerated, as $\tdun{$1$}\tdun{$2$}-\tddeux{$1$}{$2$}-\tddeux{$2$}{$1$}$ is in its kernel. 
The first values of the rank of the pairing in dimension $n$ is given by the following array:
$$\begin{array}{|c|c|c|c|c|}
\hline n&1&2&3&4\\
\hline Rk(\langle-,-\rangle_{\mid (\h_\T)_n})&1&3&16&111\\
\hline dim(Ker(\langle-,-\rangle_{\mid (\h_\T)_n}))&0&1&13&244\\
\hline \end{array}$$

\section{Ribbon basis}

\subsection{Definition}

The set $\TT_n$ of topologies on $[n]$ is partially ordered by the refinement of topologies: if $\T,\T'\in \TT_n$, $\T\leq \T'$ if any open set of $\T$
is an open set of $\T'$. For example, here is the Hasse graph of  $\TT_2$:
$$
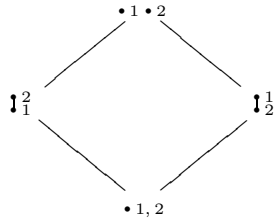
\xymatrix{&\tdun{$1$}\tdun{$2$}\ar@{-}[rd]\ar@{-}[ld]&\\
\tddeux{$1$}{$2$}\ar@{-}[rd]&&\tddeux{$2$}{$1$} \ar@{-}[ld]\\
&\tdun{$1,2$}\hspace{2mm}&}$$

\begin{defi} 
We define a basis $(R_\T)_{\T\in \TT}$ of $\h_\TT$ in the following way: for all $\T\in \TT_n$, $n\geq 0$,
$$\T=\sum_{\T'\leq \T} R_{\T'}.$$
\end{defi}

{\bf Examples.} If $\{i,j\}=\{1,2\}$:
\begin{align*}
R_{\tdun{$i,j$}\hspace{2mm}}&=\tdun{$i,j$}\hspace{2mm}, &
R_{\tddeux{$i$}{$j$}}&=\tddeux{$i$}{$j$}-\tdun{$i,j$}\hspace{2mm},&
R_{\tdun{$i$}\tdun{$j$}}&=\tdun{$i$}\tdun{$j$}-\tddeux{$i$}{$j$}-\tddeux{$j$}{$i$}+\tdun{$i,j$}\hspace{2mm}.
\end{align*}
If $\{i,j,k\}=\{1,2,3\}$:
\begin{align*}
R_{\tdun{$i,j,k$}\hspace{5mm}}&=\tdun{$i,j,k$}\hspace{5mm},\\
R_{\tddeux{$i,j$}{$k$}\hspace{3mm}}&=\tddeux{$i,j$}{$k$}\hspace{3mm}-\tdun{$i,j,k$}\hspace{5mm},\\
R_{\tddeux{$i$}{$j,k$}\hspace{3mm}}&=\tddeux{$i$}{$j,k$}\hspace{3mm}-\tdun{$i,j,k$}\hspace{5mm},\\
R_{\tdun{$i$}\tdun{$j,k$}\hspace{2mm}}&=\tdun{$i$}\tdun{$j,k$}\hspace{2mm}-\tddeux{$i$}{$j,k$}\hspace{2mm}
-\tddeux{$j,k$}{$i$}\hspace{2mm}+\tdun{$i,j,k$}\hspace{5mm},\\
R_{\tdtroisdeux{$i$}{$j$}{$k$}}&=\tdtroisdeux{$i$}{$j$}{$k$}-\tddeux{$i$}{$j,k$}\hspace{2mm}
-\tddeux{$i,j$}{$k$}\hspace{2mm}+\tdun{$i,j,k$}\hspace{5mm},\\
R_{\tdtroisun{$i$}{$k$}{$j$}}&=\tdtroisun{$i$}{$k$}{$j$}-\tdtroisdeux{$i$}{$j$}{$k$}-\tdtroisdeux{$i$}{$k$}{$j$}
+\tddeux{$i$}{$j,k$}\hspace{2mm},\\
R_{\pdtroisun{$i$}{$j$}{$k$}}&=\pdtroisun{$i$}{$j$}{$k$}-\tdtroisdeux{$k$}{$j$}{$i$}-\tdtroisdeux{$j$}{$k$}{$i$}
+\tddeux{$j,k$}{$i$}\hspace{2mm},\\
R_{\tddeux{$i$}{$j$}\tdun{$k$}}&=\tddeux{$i$}{$j$}\tdun{$k$}-\tdtroisun{$i$}{$k$}{$j$}-\pdtroisun{$j$}{$i$}{$k$}
+\tdtroisdeux{$i$}{$k$}{$j$}-\tdun{$i,j$}\hspace{2mm}\tdun{$k$}
+\tddeux{$i,j$}{$k$}\hspace{2mm}+\tddeux{$k$}{$i,j$}\hspace{2mm}-\tdun{$i,j,k$}\hspace{3mm},\\
R_{\tdun{$i$}\tdun{$j$}\tdun{$k$}}&=\tdun{$i$}\tdun{$j$}\tdun{$k$}-\tddeux{$i$}{$j$}\tdun{$k$}-
\tddeux{$i$}{$k$}\tdun{$j$}-\tddeux{$j$}{$i$}\tdun{$k$}-\tddeux{$j$}{$k$}\tdun{$i$}-\tddeux{$k$}{$i$}\tdun{$j$}-\tddeux{$k$}{$j$}\tdun{$i$}\\
&-\tddeux{$i$}{$j,k$}\hspace{2mm}-\tddeux{$j$}{$i,k$}\hspace{2mm}-\tddeux{$k$}{$i,j$}\hspace{2mm}
-\tddeux{$j,k$}{$i$}\hspace{2mm}-\tddeux{$i,k$}{$j$}\hspace{2mm}-\tddeux{$i,j$}{$k$}\hspace{2mm}\\
&+\tdtroisun{$i$}{$k$}{$j$}+\tdtroisun{$j$}{$k$}{$i$}+\tdtroisun{$k$}{$j$}{$i$}
+\pdtroisun{$i$}{$j$}{$k$}+\pdtroisun{$j$}{$i$}{$k$}+\pdtroisun{$k$}{$i$}{$j$}+2 \tdun{$i,j,k$}\hspace{5mm}.
\end{align*}

\subsection{The products and the coproduct in the basis of ribbons}

{\bf Notations.}  \begin{enumerate}
\item Let $\T\in \TT_n$ and let $I,J\subseteq [n]$. We shall write $I\leq_\T J$ if for all $i\in I$, $j\in J$, $i\leq_\T j$.
\item Let $\T\in \TT_n$ and let $I,J\subseteq [n]$. We shall write $I<_\T J$ if for all $i\in I$, $j\in J$, $i<_\T j$.
\end{enumerate}

\begin{theo}\label{theo12}\begin{enumerate}
\item Let $\T\in \TT_k$, $\T'\in \TT_l$, $k,l\geq 0$. Then:
$$R_\T.R_{\T'}=\sum_{\substack{\T'' \in \TT_{k+l},\\ \T''_{\mid [k]}=\T,\: Std(\T''_{\mid[k+l]\setminus [k]})=\T'}} R_{\T''}.$$
\item Let $\T\in \TT_k$, $\T'\in \TT_l$, $k,l\geq 0$. Then:
$$R_\T\downarrow R_{\T'}=\sum_{\substack{\T'' \in \TT_{k+l},\\ \T''_{\mid [k]}=\T, \:Std(\T''_{\mid[k+l]\setminus [k]})=\T',
\\ [k]\leq_{\T''}[k+l]\setminus [k]}} R_{\T''}.$$
\item For all $\T\in \TT_n$, $n\geq 0$:
$$\Delta(R_\T)=\sum_{\substack{O\in \T,\\ [n]\setminus O<_\T O}} R_{Std(\T_{\mid [n]\setminus O})}\otimes R_{Std(\T_{\mid O})}.$$
\end{enumerate}\end{theo}

\begin{proof} 1. {\it First step.} We first prove that for any $\T''\in \TT_{k+l}$, $\T''\leq \T.\T'$ if, and only if,
$\T''_{\mid [k]}\leq \T$ and $Std(\T''_{\mid[k+l]\setminus [k]})\leq \T'$. \\

$\Longrightarrow$. As $\T''\leq \T.\T'$, we obtain $\T''_{\mid [k]} \leq (\T.\T')_{\mid [k]}=\T$, and $Std(\T''_{\mid [k+l]\setminus [k]})\leq 
Std((\T.\T')_{\mid [k+l]\setminus [k]})=\T'$.

$\Longleftarrow$. Let $I$ be an open set of $\T''$. Then $I_1=I\cap [k]$ is an open set of $\T''_{\mid [k]}$, so $I_1$ is an open
set of $\T$. Moreover, $I_2=I\cap ([k+l] \setminus [k])(-k)$ is an open set of $Std(\T''_{\mid[k+l]\setminus [k]})$,
so $I_2$ is an open set of $\T'$. By definition of $\T.\T'$, $I_1\sqcup I_2[k]=I$ is an open set of $\T.\T'$, so $\T''\leq \T.\T'$.\\

{\it Second step.} We define a product $\star$ on $\h_\TT$ by the formula:
$$R_\T\star R_{\T'}=\sum_{\substack{\T'' \in \TT_{k+l},\\ \T''_{\mid [k]}=\T, Std(\T''_{\mid[k+l]\setminus [k]})=\T'}} R_{\T''},$$
for any $\T\in \TT_k$, $\T'\in \TT_l$, $k,l\geq 0$. Then:
\begin{align*}
\T\star \T'&=\sum_{\SS \leq \T,\SS'\leq \T'}R_\SS\star R_{\SS'}\\
&=\sum_{\SS \leq \T,\SS'\leq \T'}\sum_{\substack{\SS'' \in \TT_{k+l},\\
\SS''_{\mid [k]}=\SS, Std(\SS''_{\mid[k+l]\setminus [k]})=\SS'}} R_{\SS''}\\
&=\sum_{\substack{\SS'' \in \TT_{k+l},\\
\SS''_{\mid [k]}\leq \T, Std(\SS''_{\mid[k+l]\setminus [k]})\leq \T'}} R_{\SS''}\\
&=\sum_{\substack{\SS'' \in \TT_{k+l},\\
\mathcal{S''}\leq \T.\T'}} R_{\SS''}\\
&=\T.\T'.
\end{align*}
We used the first step for the fourth equality. So $\star=.$. \\

2. {\it First step.} We first prove that for any $\T''\in \TT_{k+l}$, $\T''\leq \T\downarrow \T'$ if, and only if,
$\T''_{\mid [k]}\leq \T$, $Std(\T''_{\mid[k+l]\setminus [k]})\leq \T'$ and $[k]\leq_{\T''} [k+l]\setminus[k]$. \\

$\Longrightarrow$. By restriction, $\T''_{\mid [k]}\leq \T$, $Std(\T''_{\mid[k+l]\setminus [k]})\leq \T'$.
Let $x\in [k]$ and $y \in [k+l]\setminus [k]$, and let $O$ be an open set of $\T''$ which contains $x$. It is also an open set of $\T\downarrow \T'$
which contains $x$, so it contains $[k+l]\setminus [k]$ by definition of $\T\downarrow \T'$. So $x \leq_{\T''} y$.

$\Longleftarrow$. As $[k]\leq_{\T''} [k+l]\setminus[k]$, any open sets of $\T''$ which contains an element of $[k]$ contains $[k+l]\setminus [k]$.
Hence, there are two types of open sets in $\T''$:
\begin{itemize}
\item Open sets $O$ contained in $[k+l]\setminus [k]$. Then $O(-k)$ is an open set of $Std(\T''_{\mid [k+l]\setminus [k]})$,
so it is an open set of $\T'$, and finally $O$ is an open set of $\T\downarrow \T'$.
\item Open sets $O$ which contain $[k+l]\setminus [k]$. Then $O\cap [k]$ is an open set of $\T''_{\mid [k]}$, so is an open set of $\T$.
Hence, $O$ is an open set of $\T\downarrow \T'$.
\end{itemize}
We obtain in this way that $\T''\leq \T\downarrow \T'$.\\

{\it Second step.} Using the first step, we conclude as in the second step of the first point. \\

3. {\it First step.} Let $O\in \T$, $\SS\leq \T_{\mid O}$ and $\mathcal{S'} \leq \T_{\mid [n]\setminus O}$.
It is not difficult to see that there exists a unique topology  $\T'\in \TT_n$ such that $\SS=\T'_{\mid O}$, 
$\mathcal{S'}=\T'_{\mid [n]\setminus O}$ and $[n]\setminus O<_{\T'} O$. It is:
$$\T'=\{\Omega \cup O\mid \Omega \in \SS'\}\cup \SS.$$

{\it Second step.} We define a coproduct on $\h_\TT$ by:
$$\delta(R_\T)=\sum_{\substack{O\in \T,\\ [n]\setminus O<_\T O}} R_{Std(\T_{\mid [n]\setminus O})}\otimes R_{Std(\T_{\mid O})},$$
for any $\T\in \TT$. Then:
\begin{align*}
\delta(\T)&=\sum_{\T'\leq \T}\delta(R_{\T'})\\
&=\sum_{\T'\leq \T} \sum_{O\in \T', \:([n]\setminus O)<_{\T'} O}
R_{Std(\T'_{\mid [n]\setminus O})}\otimes R_{Std(\T'_{\mid O})}\\
&=\sum_{O\in \T} \sum_{\substack{\T'\leq \T,\:O\in \T',\\
([n]\setminus O) <_{\T'} O}} R_{Std(\T'_{\mid [n]\setminus O})}\otimes R_{Std(\T'_{\mid O})}\\
&=\sum_{O\in \T}\sum_{\substack{\SS \leq Std(\T_{\mid [n]\setminus O}), \\\mathcal{S'}\leq Std(\T_{\mid O})}}
R_{\SS}\otimes R_{\mathcal{S'}}\\
&=\sum_{O\in \T}Std(\T_{\mid [n]\setminus O})\otimes Std(\T_{\mid O})\\
&=\Delta(\T).
\end{align*}
We used the first step for the third equality. So $\delta=\Delta$. \end{proof}

\section{Generalized T-partitions}

\label{section4}

\subsection{Definition}

\begin{defi} \label{defi13}
Let $\T \in \TT_n$. 
\begin{enumerate}
\item A generalized T-partition of $\T$ is a surjective map $f:[n] \longrightarrow [p]$ such that if $i \leq_\T j$ in $[n]$, then $f(i) \leq f(j)$ in $[p]$. 
If $f$ is a generalized T-partition of $\T$, we shall represent it by the packed word $f(1)\ldots f(n)$.
\item Let $f$ a generalized T-partition of $\T$. We shall say that $f$ is a (strict) T-partition if for all $i,j\in [n]$:
\begin{itemize}
\item $i<_\T j$ and $i>j$ implies that $f(i)<f(j)$ in $[p]$.
\item If $i<j<k$, $i\sim_\T k$ and $f(i)=f(j)=f(k)$, then $i\sim_\T j$ and $j\sim_\T k$.
\end{itemize}
\item The set of generalized T-partitions of $\T$ is denoted by $\P(\T)$; the set of (strict) T-partitions of $\T$ is denoted by $\P_s(\T)$.
\item If $f \in \P(\T)$, we put:
\begin{align*}
\ell_1(f)&=\sharp\{(i,j)\in [n]^2\mid i<_\T j,\: i<j,\:\mbox{ and }f(i)=f(j)\},\\
\ell_2(f)&=\sharp\{(i,j)\in [n]^2\mid i<_\T j, \:i>j,\:\mbox{ and }f(i)=f(j)\},\\
\ell_3(f)&=\sharp\{(i,j,k)\in [n]^3\mid i<j<k,\: i\sim_\T k, i\notsim j, \: j\notsim k\mbox{ and }f(i)=f(j)=f(k)\}.
\end{align*}
Note that $f$ is strict if, and only if, $\ell_2(f)=\ell_3(f)=0$.
\end{enumerate}\end{defi}

{\bf Example.} Let $\T=\tdtroisun{$2,4$}{$5$}{$1$}\hspace{2mm} \tdun{$3$}$. 
If $f$ is a packed word of length $5$, it is a generalized $T$-partition of $\T$ if, and only if, $f(2)=f(4)\leq f(1),f(5)$. Hence:
$$\P(\T)=\left\{\begin{array}{c}
(11111),(11112),(11211),(11212),(11213),(11312),(21111),\\
(21112),(21113),(21211),(21212),(21213),(21311),(21312),\\
(21313),(21314),(21413),(22122),(22123),(31112),(31211),\\
(31212),(31213),(31214),(31312),(31412),(32122),(32123),\\
(32124),(41213),(41312),(42123)
\end{array}\right\}.$$
Moreover, $f$ is a strict $T$-partition of $\T$ if, and only if, $f(2)=f(4)<f(1)$, $f(2)=f(4)\leq f(5)$, and $f(3)\neq f(2),f(4)$. Hence:
$$\P_s(\T)=\left\{\begin{array}{c}
(21211),(21212),(21213),(21311),(21312),(21313),(21314),\\
(21413),(31211),(31212),(31213),(31214),(31312),(31412),\\
(32122),(32123),(32124),(41213),(41312),(42123)
\end{array}\right\}.$$

{\bf Remarks.} \begin{enumerate}
\item Let $f\in \P(\T)$. If $i\sim_\T j$ in $[n]$, then $i \leq_\T j$ and $j \leq_\T i$, so $f(i)\leq f(j)$ and $f(j) \leq f(i)$, and finally $f(i)=f(j)$.
\item If $\T \in \TT$ is $T_0$, then the set of strict T-partitions of $\T$ is the set of P-partitions of the poset $\T$, as defined in \cite{Gessel,Stanley}.
We shall now omit the term "strict" and simply write "T-partitions".
\end{enumerate}

\begin{prop} \label{prop14}
Let $q=(q_1,q_2,q_3)\in \K^3$. We define a linear map $\Gamma_q:\h_\TT\longrightarrow \WQSym$ in the following way: for all $\T\in \TT_n$, $n\geq 0$,
$$\Gamma_q(\T)=\sum_{f\in \P(\T)} q_1^{\ell_1(f)}q_2^{\ell_2(f)}q_3^{\ell_3(f)}f(1)\ldots f(n).$$
Then $\Gamma_q$ is a homogeneous surjective Hopf algebra morphism. 
Moreover, $j\circ \Gamma_{(q_2,q_1,q_3)})=\Gamma_{(q_1,q_2,q_3)}\circ \iota$.
\end{prop}

\begin{proof} We shall use the following notations: if  $\T\in \TT$ and $f\in \P(\T)$, we put $\ell(f)=(\ell_1(f),\ell_2(f),\ell_3(f))$ and
$q^{\ell(f)}=q_1^{\ell_1(f)}q_2^{\ell_2(f)}q_3^{\ell_3(f)}$. \\

Let us first prove that  $\Gamma_q$ is an algebra morphism.
Let $\T,\T'\in \TT$, of respective degree $n$ and $n'$. Let $f=f(1)\ldots f(n+n')$ be a packed word.
If $1\leq i\leq n<j\leq n+n'$, then $i$ and $j$ are not comparable for $\leq_{\T.\T'}$. Hence, $f \in \P(\T.\T')$ if, and only if,
 $f'=Pack(f(1)\ldots f(n))\in \P(\T)$ and $f''=Pack(f(n+1)\ldots f(n+n')) \in \P(\T')$. Moreover, $\ell(f)=\ell(f')+\ell(f'')$, so:
\begin{align*} 
\Gamma_q(\T.\T')&=\sum_{f'\in \P(\T),f''\in \P(\T')} \sum_{\substack{Pack(f(1)\ldots f(n))=f'\\ Pack(f(n+1)\ldots f(n+n'))=f''}} q^{\ell(f)}f(1)\ldots f(n+n')\\
&=\sum_{f'\in \P(\T),f''\in \P(\T')}q^{\ell(f')}q^{\ell(f'')} \sum_{\substack{Pack(f(1)\ldots f(n))=f'\\ 
Pack(f(n+1)\ldots f(n+n'))=f''}}f(1)\ldots f(n+n')\\
&=\sum_{f'\in \P(\T),f''\in \P(\T')}  q^{\ell(f')}q^{\ell(f'')}f'. f''\\
&=\Gamma_q(\T).\Gamma_q(\T').
 \end{align*}
 
Let us now prove that $\Gamma_q$ is a coalgebra morphism. Let $\T \in \TT_n$. We consider the two following sets:
\begin{align*}
A&=\bigsqcup_{O\in \T} \P(\T_{\mid [n]\setminus O})\times \P(\T_{\mid O}),&B&=\{(f,i)\mid f\in \P(\T), 0\leq i \leq \max(f)\}.
\end{align*}
Let $(f,g) \in \P(\T_{\mid [n]\setminus O})\times \P(\T_{\mid O}) \subseteq A$. We define $h:[n]\longrightarrow \N$ 
by $h(i)=f(i)$ if $i\notin O$  and $h(i)=g(i)+\max(f)$ if $i\in O$. As $f$ and $g$ are packed words, $h(1)\ldots h(n)$ also is.
Let us assume that $i \leq_\T j$. Three cases are possible:
\begin{itemize}
\item $i,j \notin O$. As $f(i)\leq f(j)$, $h(i) \leq h(j)$.
\item $i,j \in O$. As $g(i)\leq g(j)$, $h(i)\leq h(j)$.
\item $i\notin O$ and $j\in O$. Then $h(i)=f(i)\leq \max(f)<\max(f)+g(j)=h(j)$.
\end{itemize}
Consequently, $h\in \P(\T)$. Hence, we define a map $\theta:A\longrightarrow B$, sending $(f,g)$ to $(h,\max(f))$.\\
 
$\theta$ is injective: if $\theta(f,g)=\theta(f',g')=(h,k)$, then $\max(f)=\max(f')=k$. Moreover:
$$O=h^{-1}(\{\max(f)+1,\ldots,\max(h)\})=h^{-1}(\{k+1,\ldots,\max(h)\})=O'.$$
Then $f=h_{\mid O}=h_{\mid O'}=f'$ and $g=Pack(h_{\mid [n] \setminus O})=Pack(h_{\mid [n] \setminus O})=g'$.
Finally, $(f,g)=(f',g')$. \\
 
$\theta$ is surjective: let $(h,k) \in B$. We put $O=h^{-1}(\{k+1,\ldots \max(h)\})$. Let $i\in O$ and $j\in [n]$, such that 
$i\leq_\T j$. As $h\in \P(\T)$, $h(j)\geq h(i)>k$, so $j\in O$ and $O$ is an open set of $\T$.
Let $f=h_{\mid [n] \setminus O}$ and $g=Pack(h_{\mid O})$. By restriction, $f\in \P(\T_{\mid [n] \setminus O})$ and 
$g\in \P(\T_{\mid O})$. Moreover, as $h$ is a packed word, $\max(f)=k$, and for all $i\in O$, $g(i)=h(i)-k$. 
This implies that $\theta(f,g)=(h,k)$.

Let $(f,g) \in \P(\T_{\mid [n]\setminus O})\times \P(\T_{\mid O}) \subseteq A$, and let $\theta(f,g)=(h,k)$. 
If $i\notin O$ and $j\in O$, $h(i)<h(j)$. Hence:
\begin{align*}
\ell_1(h)&=\sharp\{(i,j)\in ([n]\setminus O)^2\mid i<_\T j,\: i<j\mbox{ and }h(i)=h(j)\}\\
&+\sharp\{(i,j)\in O^2\mid i<_\T j,\:i<j\mbox{ and }h(i)=h(j)\}\\
&=\sharp\{(i,j)\in ([n]\setminus O)^2\mid i<_\T j,\:i<j\mbox{ and }f(i)=f(j)\}\\
&+\sharp\{(i,j)\in O^2\mid i<_\T j,\:i<j\mbox{ and }g(i)+\max(f)=g(j)+\max(f)\}\\
&=\ell_1(f)+\ell_1(g).
\end{align*}
Similarly, $\ell_2(h)=\ell_2(f)+\ell_2(g)$ and $\ell_3(h)=\ell_3(f)+\ell_3(g)$. We obtain:
\begin{align*}
\Delta\circ \Gamma_q(\T)&=\sum_{h\in \P(\T)}\sum_{0\leq k\leq \max(h)}
q^{\ell(h)} h_{\mid h^{-1}(\{1,\ldots,k\})}\otimes Pack\left(h_{\mid h^{-1}(\{k+1,\ldots,\max(h)\})}\right)\\
&=\sum_{(h,k)\in B}
q^{\ell(h)} h_{\mid h^{-1}(\{1,\ldots,k\})}\otimes Pack\left(h_{\mid h^{-1}(\{k+1,\ldots,\max(h)\})}\right)\\
&=\sum_{(f,g)\in A}q^{\ell(f)}q^{\ell(g)} f\otimes g\\
&=\sum_{O\in \T}\left(\sum_{f\in \P(\T_{\mid [n]\setminus O})}
q^{\ell(f)}f\right)\otimes \left(\sum_{g\in \P(\T_{\mid O})} q^{\ell(g)}g\right)\\
&=(\Gamma_q \otimes \Gamma_q)\circ \Delta(\T).
\end{align*}
Consequently, $\Gamma_q$ is a Hopf algebra morphism.\\

Let $f$ be a packed word of length $n$, and $\T_f$ be the associated topology, as defined in section \ref{section2}. 
Note that $f \in \P(\T_f)$; moreover, if $g \in \P(\T_f)$ is different from $f$, then $\max(g)<\max(f)$. Hence:
$$\Gamma_q(\T_f)=f+\left(\mbox{linear span of packed words $g$ of maximum $<max(f)$}\right).$$
In particular, if $f=(1\ldots 1)$, $\T_w=\tdun{$1,\ldots,n $}\hspace{7.5mm}$, and $\Gamma_q(T_f)=(1\ldots 1)$.
By a triangularity argument, $\Gamma_q$ is surjective. \\

Let $\T \in \TT$. It is not difficult to prove that $\P(\T)=j\left(\P(\iota(\T))\right)$. Moreover, if $f\in \P(\iota(\T))$ and $g=j(f)$, then 
$\ell_1(f)=\ell_2(g)$, $\ell_2(f)=\ell_1(g)$ and $\ell_3(f)=\ell_3(g)$. So:
$$j\circ \Gamma_{(q_2,q_2,q_3)}\circ \iota(\T)=\sum_{g \in \P(\T)} q_2^{\ell_2(g)} q_1^{\ell_1(g)} q_2^{\ell_2(g)} g
=\Gamma_{(q_1,q_2,q_3)}(\T),$$
so $j\circ \Gamma_{(q_2,q_2,q_3)}\circ \iota=\Gamma_{(q_1,q_2,q_3)}$. \end{proof}

{\bf Examples.}
\begin{align*}
\Gamma_q(\tdun{$1$})&=(1),\\
\Gamma_q(\tdun{$1$}\tdun{$2$})&=(12)+(21)+(11),\\
\Gamma_q(\tddeux{$1$}{$2$})&=(12)+q_1(11),\\
\Gamma_q(\tddeux{$2$}{$1$})&=(21)+q_2(11),\\
\Gamma_q(\tdun{$1,2$}\hspace{3mm})&=(11),\\
\Gamma_q(\tdtroisun{$1$}{$3$}{$2$})&=(123)+(132)+(122)+q_1(112)+q_1(121)+q_1^2(111),\\
\Gamma_q(\tdtroisun{$2$}{$3$}{$1$})&=(213)+(312)+(212)+q_2(112)+q_1(211)+q_1q_2(111),\\
\Gamma_q(\tdtroisun{$3$}{$2$}{$1$})&=(231)+(321)+(221)+q_2(121)+q_2(211)+q_2^2(111),\\
\Gamma_q(\tddeux{$1,2$}{$3$}\hspace{3mm})&=(112)+q_1^2(111),\\
\Gamma_q(\tddeux{$1,3$}{$2$}\hspace{3mm})&=(121)+q_1q_2q_3(111),\\
\Gamma_q(\tddeux{$2,3$}{$1$}\hspace{3mm})&=(211)+q_2^2(111),\\
\Gamma_q(\tddeux{$1$}{$2,3$}\hspace{3mm})&=(122)+q_1^2(111),\\
\Gamma_q(\tddeux{$2$}{$1,3$}\hspace{3mm})&=(212)+q_1q_2q_3(111),\\
\Gamma_q(\tddeux{$3$}{$1,2$}\hspace{3mm})&=(221)+q_2^2(111).
\end{align*}

{\bf Remarks.} \begin{enumerate}
\item In particular:
\begin{align*}
\Gamma_{(1,1,1)}(\T)&=\sum_{f\in \P(\T)} f(1)\ldots f(n),&
\Gamma_{(1,0,0)}(\T)&=\sum_{f\in \P_s(\T)} f(1)\ldots f(n).
\end{align*}
\item The restriction of $\Gamma_{(1,0,0)}$ to $\h_\SP$ is the map $\Gamma$ defined in section \ref{section1.2}.
\end{enumerate}

\subsection{Linear extensions}

\begin{defi} \label{defi15}
Let $\T\in \TT_n$, $n\geq 0$. A linear extension of $\T$ is an ordered partition $A=(A_1,\ldots,A_k)$ of $[n]$ such that:
\begin{itemize}
\item the equivalence classes of $\sim_\T$ are $A_1,\ldots,A_k$;
\item if, in the poset $\overline{\T}$, $A_i\leq_\T A_j$, then $i\leq j$.
\end{itemize}
The set of linear extensions of $\T$ is denoted by $\L(\T)$.
\end{defi}

{\bf Notations.} If $A=(A_1,\ldots,A_k)$ is a linear extension of $\T \in \TT_n$, we bijectively  associate to $A$ the packed word $f=
f(1)\ldots f(n)$  of length $n$ such that for all $i\in [n]$, for all $j\in [k]$, $f(i)=j$ if, and only if $i\in A_j$. 
We shall now use the packed word $f$ instead of $A$.
\\

{\bf Example.} Let $\T=\tdtroisun{$2,4$}{$1,5,6$}{$3$}\hspace{4mm}$.  The linear extensions of $\T$ are
$(\{2,4\},\{3\},\{1,5,6\})$ and $(\{2,4\},\{1,5,6\},\{3\})$, or, expressed in packed words, $(312133)$ and $(213122)$.\\

{\bf Remark.} Let $\T\in \TT_n$, $n \geq 0$, and $f=f(1)\ldots f(n)$ be a packed word of length $n$. Then $f\in \L(\T)$ if, and only if:
\begin{enumerate}
\item for all $i,j \in [n]$, $f(i)=f(j)$ if, and only if, $i\sim_\T j$;
\item for all $i,j \in [n]$, $i<_\T j$ implies that $f(i)<f(j)$.
\end{enumerate}
Hence, $\L(\T) \subseteq P(\T)$ and more precisley:
$$\L(\T)=\{f\in \P(\T)\mid \max(f)=\sharp \overline{\T}\}=\{f\in \P_s(\T)\mid \max(f)=\sharp \overline{\T}\}.$$

\begin{prop}\label{prop16}\begin{enumerate}
\item Let $f',f''$ be two packed words, of respective length $n$ and $n'$. We put:
$$f'\shuffle f''= \sum_{\substack{Pack(f(1)\ldots f(n))=f'\\ Pack(f(n+1)\ldots f(n+n'))=f''\\ 
\{f(1),\ldots f(n)\}\cap \{f(n+1),\ldots, f(n+n')\}=\emptyset}} f.$$
This defines a product on $\WQSym$, such that $(\WQSym,\shuffle ,\Delta)$ is a Hopf algebra.
\item Let $L$ be the following map:
$$L:\left\{\begin{array}{rcl}
\h_\TT&\longrightarrow&\WQSym\\
\T&\longrightarrow&\displaystyle \sum_{f\in \L(\T)}f.
\end{array}\right.$$
Then $L$ is a surjective Hopf morphism from $\h_\TT$ to $(\WQSym,\shuffle ,\Delta)$. Moreover, $j\circ L=L\circ \iota$.
\end{enumerate}\end{prop}

\begin{proof} Let $\T,\T'\in \TT$, of respective degree $n$ and $n'$. Let $f=f(1)\ldots f(n+n')$ be a packed word.
If $1\leq i\leq n<j\leq n+n'$, then $i$ and $j$ are not comparable for $\leq_{\T.\T'}$. Hence, $f \in \L(\T.\T')$ if, and only if:
\begin{itemize}
\item $f'=Pack(f(1)\ldots f(n))\in \L(\T)$ and $f''=Pack(f(n+1)\ldots f(n+n')) \in \L(\T')$. 
\item $\{f(1),\ldots f(n)\}\cap \{f(n+1),\ldots,f(n+n')\}=\emptyset$.
\end{itemize}
So:
\begin{align*} 
L(\T.\T')&=\sum_{f'\in \L(\T),f''\in \L(\T')} \sum_{\substack{Pack(f(1)\ldots f(n))=f'\\ Pack(f(n+1)\ldots f(n+n'))=f''\\
\{f(1),\ldots f(n)\}\cap \{f(n+1),\ldots,f(n+n')\}=\emptyset}}
f(1)\ldots f(n+n')\\
&=\sum_{f'\in \L(\T),f''\in \L(\T')} f\shuffle f'\\
&=L(\T)\shuffle L(\T').
 \end{align*}
 
 Let $\T \in \TT_n$. We consider the two following sets:
 \begin{align*}
 A&=\bigsqcup_{O\in \T} \L(\T_{\mid [n]\setminus O})\times \L(\T_{\mid O}),&B&=\{(f,i)\mid f\in \L(\T), 0\leq i \leq \max(f)\}.
 \end{align*}
 Let $(f,g) \in \L(\T_{\mid [n]\setminus O})\times \L(\T_{\mid O}) \subseteq A$. We define $h:[n]\longrightarrow \N$ 
by $h(i)=f(i)$ if $i\notin O$  and $h(i)=g(i)+\max(f)$ if $i\in O$. As $f$ and $g$ are packed words, $h(1)\ldots h(n)$ also is.
Moreover, as $O$ and $[n]\setminus O$ are union of equivalence classes of $\sim_\T$, for all $i,j \in [n]$, $h(i)=h(j)$ if, and only if, $i\sim_\T j$.
 Let us assume that $i \leq_\T j$. Three cases are possible:
 \begin{itemize}
 \item $i,j \notin O$. As $f(i)\leq f(j)$, $h(i) \leq h(j)$.
 \item $i,j \in O$. As $g(i)\leq g(j)$, $h(i)\leq h(j)$.
 \item $i\notin O$ and $j\in O$. Then $h(i)=f(i)\leq \max(f)<\max(f)+g(j)=h(j)$.
 \end{itemize}
 Consequently, $h\in \L(\T)$. Hence, we define a map $\theta:A\longrightarrow B$, sending $(f,g)$ to $(h,\max(f))$.
The proof of the bijectivity of $\theta$ is similar to the proof for the case of generalized T-partitions.  We obtain:
\begin{align*}
\Delta\circ L(\T)&=\sum_{h\in \L(\T)}\sum_{0\leq k\leq \max(h)}
 h_{\mid h^{-1}(\{1,\ldots,k\})}\otimes Pack\left(h_{\mid h^{-1}(\{k+1,\ldots,\max(h)\})}\right)\\
&=\sum_{(h,k)\in B}
 h_{\mid h^{-1}(\{1,\ldots,k\})}\otimes Pack\left(h_{\mid h^{-1}(\{k+1,\ldots,\max(h)\})}\right)\\
&=\sum_{(f,g)\in A} f\otimes g\\
&=\sum_{O\in \T}\left(\sum_{f\in \L(\T_{\mid [n]\setminus O})}f\right)\otimes \left(\sum_{g\in \L(\T_{\mid O})} g\right)\\
&=(L \otimes L)\circ \Delta(\T).
\end{align*}

Let $f$ be a packed word. Then $\L(\T_f)=\{f\}$, so $L(\T_f)=f$, that is $L$ is surjective.
As $L$ is compatible with the products and the coproducts of both $\h_\T$ and $(\WQSym,\shuffle ,\Delta)$, its image, that is to say  
$(\WQSym,\shuffle ,\Delta)$, is a Hopf algebra, and $L$ is a Hopf algebra morphism. 

It is not difficult to prove that for all $\T\in \TT$, $j(\L(\iota(\T)))=\L(\T)$, which implies that $j\circ L \circ \iota=L$. \end{proof}

{\bf Remark.} The product $\shuffle$ is defined and used in  \cite{FPT}. \\

{\bf Examples.}
\begin{align*}
L(\tdun{$1$})&=(1),&L(\tdun{$1$}\tdun{$2$})&=(12)+(21),&L(\tddeux{$1$}{$2$})&=(12),\\
L(\tddeux{$2$}{$1$})&=(21),&L(\tdun{$1,2$}\hspace{3mm})&=(11),&L(\tdun{$1,2,3$}\hspace{5mm})&=(111),\\
L(\tddeux{$1,2$}{$3$}\hspace{3mm})&=(112),&L(\tddeux{$1,3$}{$2$}\hspace{3mm})&=(121),&L(\tddeux{$2,3$}{$1$}\hspace{3mm})&=(211),\\
L(\tddeux{$3$}{$1,2$}\hspace{3mm})&=(221),&L(\tddeux{$2$}{$1,3$}\hspace{3mm})&=(212),&L(\tddeux{$1$}{$2,3$}\hspace{3mm})&=(122),\\
L(\tdtroisun{$1$}{$3$}{$2$})&=(123)+(132),&L(\tdtroisun{$2$}{$3$}{$1$})&=(213)+(312),&L(\tdtroisun{$3$}{$2$}{$1$})&=(231)+(321),\\
L(\tdtroisdeux{$1$}{$2$}{$3$})&=(123),&L(\tdtroisdeux{$2$}{$1$}{$3$})&=(213),&L(\tdtroisdeux{$3$}{$1$}{$2$})&=(231),\\
L(\tdtroisdeux{$1$}{$3$}{$2$})&=(132),&L(\tdtroisdeux{$2$}{$3$}{$1$})&=(312),&L(\tdtroisdeux{$3$}{$2$}{$1$})&=(321).
\end{align*}
Here is an example of product $\shuffle$:
\begin{align*}
(112)\shuffle(12)&=(11234)+(11324)+(11423)+(22314)+(22413)+(33412).
\end{align*}

\subsection{From linear extensions to T-partitions}

\begin{defi} \label{defi17}
Let $f,g$ be two packed words of the same length $n$. We shall say that $g\leq f$ if:
\begin{itemize}
\item For all $i,j \in [n]$, $f(i)\leq f(j)$ implies that $g(i)\leq g(j)$.
\item For all $i,j \in [n]$, $f(i)>f(j)$ and $i<j$ implies that $g(i)>g(j)$.
\item For all $i,j \in [n]$, $f(i)=f(j)$ implies that $g(i)=g(j)$.
\end{itemize}\end{defi}
The relation $\leq$ is a partial order on the set of packed words of length $n$. Here are the Hasse graphs of these posets if $n=2$ or $n=3$:
$$
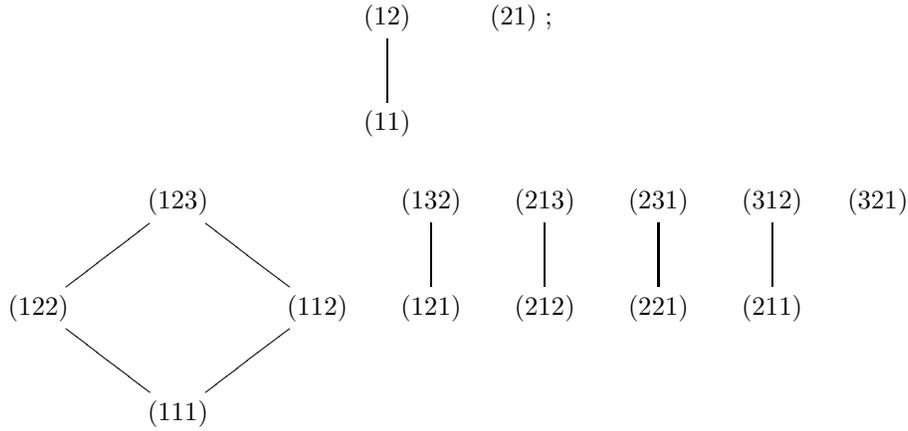
\xymatrix{(12)\ar@{-}[d]&(21)\\ (11)};$$
$$\xymatrix{&(123)\ar@{-}[rd] \ar@{-}[ld] &\\(122)\ar@{-}[rd] &&(112)\ar@{-}[ld]\\&(111)&}\hspace{.5cm}
\xymatrix{(132)\ar@{-}[d]\\ (121)}\hspace{.5cm}\xymatrix{(213)\ar@{-}[d]\\ (212)}\hspace{.5cm}
\xymatrix{(231)\ar@{-}[d]\\ (221)}\hspace{.5cm}\xymatrix{(312)\ar@{-}[d]\\ (211)}\hspace{.5cm}(321)$$
 
 {\bf Remark.} This order is also introduced and used in \cite{Maurice}.
 
\begin{lemma}\label{lemma18}
Let $f,g$ be packed words of length $n$. Then $g\leq f$ if, and only if, $g$ is a T-partition of $\T_f$.
\end{lemma}

\begin{proof}   $\Longrightarrow$. Let us assume that $g\leq f$. If $i\leq_{\T_f} j$, then $f(i)\leq f(j)$, so $g(i)\leq g(j)$.
If $i<_{\T_f} j$ and $i>j$, then $f(i)<f(j)$, so $g(i)<g(j)$. We assume that $i<j<k$, $i\sim_{\T_f} k$ and $g(i)=g(j)=g(k)$.
As $i\sim_\T k$,  $f(i)=f(k)$. If $f(j)<f(i)$, as $g\leq f$, we obtain $g(j)<g(i)$: contradiction. So $f(i)\leq f(j)$.
If $f(j)>f(k)$, as $g\leq f$, we obtain $g(j)>g(k)$: contradiction. So $f(j)\leq f(k)$, and finally, $f(i)=f(j)=f(k)$,
so $i\sim_{\T_f} j$ and $j\sim_{\T_f} k$. Hence, $g\in \P_s(\T_f)$.\\

$\Longleftarrow$. Let us assume that $g\in \P_s(\T_f)$. If $f(i)\leq f(j)$, then $i\leq_{\T_f} j$, so $g(i) \leq g(j)$.
If $f(i)=f(j)$, then $i\sim_{\T_f} j$, so $g(i)=g(j)$. If $f(i)<f(j)$ and $i>j$, then $i<_{\T_f} j$ and $i>j$, so $g(i)<g(j)$.
This implies $g \leq f$. \end{proof}

\begin{prop}\label{prop19}
We define:
$$\varphi_{(1,0,0)}:\left\{\begin{array}{rcl}
\WQSym&\longrightarrow&\WQSym\\
f&\longrightarrow&\displaystyle \sum_{g\leq f} g.
\end{array}\right.$$
Then $\varphi_{(1,0,0)}$ is a Hopf algebra isomorphism from $(\WQSym,\shuffle ,\Delta)$ to $(\WQSym,.,\Delta)$ such that 
the following diagram commutes:
$$\xymatrix{\h_\TT\ar[r]^{L\hspace{15mm}}\ar[rd]_{\Gamma_{(1,0,0)}}&(\WQSym,\shuffle ,\Delta)\ar[d]^{\varphi_{(1,0,0)}}\\
&(\WQSym,.,\Delta)}$$
\end{prop}

\begin{proof} Let $\T\in \TT_n$, $n \geq 0$. \\

{\it First step.} Let $f\in \L(\T)$ and $g\leq f$; let us prove that $g\in \P_s(\T)$. This is the generalization of lemma \ref{lemma18} to any finite topology.
\begin{enumerate}
\item If $i\leq_\T j$, then $f(i)\leq f(j)$, so $g(i) \leq g(j)$. 
\item If $i<_\T j$ and $j<i$, then $f(i)<f(j)$, so $g(i)<g(j)$. 
\item Let us assume that $i<j<k$, $i\sim_\T k$ and $g(i)=g(j)=g(k)$. If $i\notsim j$, then $f(i)\neq f(j)$. If $f(i)>f(j)$, as $g\leq f$,
we obtain $g(i)>g(j)$: this is a contradiction. So $f(i)<f(j)$. As $\sim_\T$ is an equivalence, $j\notsim k$, and we obtain in the same way
$f(j)<f(k)$. So $f(i)<f(k)$, and $i\notsim k$: this is a contradiction. As a consequence, $i\sim_\T j$ and $j\sim_\T k$.
\end{enumerate}
This proves that $g\in \P_s(\T)$. \\

{\it Second step.} Let us now consider an element $g\in \P_s(\T)$. We want to prove that there exists a unique $f\in \L(\T)$ such that $g\leq f$. 
For all $p\in [\max(g)]$, $g^{-1}(\{p\})$ is the union of equivalence classes of $\sim_\T$, and we put:
$$g^{-1}(\{p\})=C_{p,1}\sqcup\ldots \sqcup C_{p,k_p}.$$
Moreover, as $g\in P_s(\T)$, necessarily the $C_{p,r}$ are intervals. Hence, we assume that for all $p$:
$$C_{p,1}<\ldots <C_{p,k_p},$$
which means for all $r<s$, all the elements of $C_{p,r}$ are smaller than all the elements of $C_{p,s}$.

{\it Unicity.} Let us assume there exists $f \in \L(\T)$, such that $g\leq f$. 
The linear extension $f$ is constant on $C_{p,r}$: we put $f(C_{p,r})=\{c_{p,r}\}$. As $f\in \L(\T)$,
the $c_{p,r}$ are all distincts. 
\begin{enumerate}
\item If $r<s$ and $c_{p,r}>c_{p,s}$, choosing $i \in C_{p,r}$ and $j \in C_{p,s}$, 
then $i<j$ and $f(i)>f(j)$, so, as $g\leq f$, $p=g(i)>g(j)=p$: contradiction. So $c_{p,r}< c_{p,s}$.
\item If $p<q$ and $c_{q,s}<c_{p,r}$, choosing $i \in C_{p,r}$ and $j\in C_{q,s}$, then $f(i)\geq f(j)$,
which implies that $p=g(i)\geq g(j)=q$: contradiction. So $c_{p,r}<c_{q,s}$.
\end{enumerate}
We finally obtain:
$$c_{1,1}<\ldots<c_{1,k_1}<\ldots<c_{\max(g),1}<\ldots <c_{\max(g),k_{\max(g)}},$$
which entirely determines $f$: $f$ is unique.

{\it Existence.} Let us consider the packed word $f$ determined by:
$$f(i)=k_1+\ldots+k_{p-1}+r \mbox{ if }x\in C_{p,r}.$$
The values of $f$ on two different subsets $C_{p,r}$ are different, so $i \sim_\T j$ if, and only if, $f(i)=f(j)$.
If $i<_\T j$, assuming  $i \in C_{p,r}$ and $j\in C_{q,s}$, then $p=g(x)\leq g(y)=q$. When $p<q$, then $f(i)<f(j)$.
When $p=q$, if $j<i$ we should have, as $g\in \P_s(\T)$, $g(i)<g(j)$: contradiction. So $i<j$, and $r<s$, so $f(x)<f(y)$. Finally, $f\in \L(\T)$. 

If $f(i)=f(j)$, then $i$ and $j$ are in the same $C_{p,r}$, so $g(i)=g(j)$. 
If $f(i)\leq f(j)$, assuming that $i \in C_{p,r}$ and $j\in C_{q,s}$, then $p\leq q$, so $p=g(x)\leq g(y)=q$. 
If $f(i)<f(j)$ and $i>j$ then $p\leq q$; if $p=q$, then $r>s$, so $f(i)>f(j)$: contradiction. So $p<q$, and $p=g(i)<g(j)=q$. We obtain that $g\leq f$. \\

As a conclusion:
$$\varphi_{(1,0,0)}\circ L(\T)=\sum_{f\in \L(\T)}\sum_{g\leq f} g(1)\ldots g(n)=\sum_{g\in \P_s(\T)}g(1)\ldots g(n)=\Gamma_{(1,0,0)}(\T).$$
So $\varphi_{(1,0,0)}\circ L=\Gamma_{(1,0,0)}$. \\

{\it Third step.} Let $x,y\in \WQSym$. As $L$ is surjective, there exists $x',y'\in \h_\T$ such that $L(x')=x$ and $L(y')=y$. Then:
\begin{align*}
\varphi_{(1,0,0)}(x\shuffle y)&=\varphi_{(1,0,0)}(L(x')\shuffle L(y'))\\
&=\varphi_{(1,0,0)}\circ L(x'.y')\\
&=\Gamma_{(1,0,0)}(x'.y')\\
&=\Gamma_{(1,0,0)}(x').\Gamma_{(1,0,0)}(y')\\
&=\varphi_{(1,0,0)}\circ L(x').\varphi_{(1,0,0)}\circ L(y')\\
&=\varphi_{(1,0,0)}(x).\varphi_{(1,0,0)}(y).
\end{align*}
Moreover:
\begin{align*}
\Delta \circ \varphi_{(1,0,0)}(x)&=\Delta \circ \varphi_{(1,0,0)}\circ L(x')\\
&=\Delta\circ \Gamma_{(1,0,0)}(x')\\
&=(\Gamma_{(1,0,0)}\otimes \Gamma_{(1,0,0)})\circ \Delta(x')\\
&=(\varphi_{(1,0,0)}\otimes \varphi_{(1,0,0)})\circ (L\otimes L)\circ \Delta(x')\\
&=(\varphi_{(1,0,0)}\otimes \varphi_{(1,0,0)})\circ \Delta(L(x'))\\
&=(\varphi_{(1,0,0)}\otimes \varphi_{(1,0,0)})\circ \Delta(x).
\end{align*}
So $\varphi_{(1,0,0)}$ is a Hopf algebra morphism.\\

Let $x\in \WQSym$. As $\Gamma_{(1,0,0)}$ is surjective, there exists $x'\in \h_\TT$, such that $\Gamma_{(1,0,0)}(x')=x$.
Then $\varphi_{(1,0,0)}(L(x'))=x$: $\varphi_{(1,0,0)}$ is surjective. Since it is homogeneous and the homogeneous components of $\WQSym$
are finite dimensional, it is an isomorphism. \end{proof}

{\bf Examples.}
$$\varphi_{(1,0,0)}((123))=(123)+(122)+(112)+(111),$$
\begin{align*}
\varphi_{(1,0,0)}((132))&=(132)+(121),&\varphi_{(1,0,0)}((213))&=(213)+(212),\\
\varphi_{(1,0,0)}((231))&=(231)+(221),&\varphi_{(1,0,0)}((312))&=(312)+(211),\\
\varphi_{(1,0,0)}((321))&=(321),&\varphi_{(1,0,0)}((112))&=(112)+(111),\\
\varphi_{(1,0,0)}((121))&=(121),&\varphi_{(1,0,0)}((211))&=(211),\\
\varphi_{(1,0,0)}((122))&=(122)+(111),&\varphi_{(1,0,0)}((212))&=(212),\\
\varphi_{(1,0,0)}((221))&=(221),&\varphi_{(1,0,0)}((111))&=(111).
\end{align*}

\begin{cor} \label{cor20}
For any finite topology $\T$ of degree $n$:
$$\P_s(\T)=\bigsqcup_{f\in \L(\T)} \{g\mid g\leq f\}=\bigsqcup_{f\in \L(\T)} \P_s(\T_f).$$
\end{cor}

\begin{proof} This is the first step of the proof of proposition \ref{prop19}. \end{proof}

So $T$-partitions only depend on linear extensions. In the case of special posets, that is to say of $T_0$ topologies, 
the previous corollary is proved in \cite{Stanley,Gessel}. \\

{\bf Remark.}  Let $\varphi_{(0,1,0)}=j \circ \varphi_{(1,0,0)} \circ j$. Then:
$$\varphi_{(0,1,0)} \circ L=j \circ \varphi_{(1,0,0)} \circ j \circ L=j\circ \varphi_{(1,0,0)} \circ L \circ \iota
=j\circ \Gamma_{(1,0,0)} \circ \iota=\Gamma_{(0,1,0)}.$$
Moreover, for all packed word $f$:
$$\varphi_{(0,1,0)}(f)=\sum_{g\leq' f} g,$$
where the order relation $\leq'$ on packed words is defined by $g\leq' f$ if:
\begin{itemize}
\item For all $i,j \in [n]$, $f(i)\leq f(j)$ implies that $g(i)\leq g(j)$.
\item For all $i,j \in [n]$, $f(i)<f(j)$ and $i<j$ implies that $g(i)<g(j)$.
\item For all $i,j \in [n]$, $f(i)=f(j)$ implies that $g(i)=g(j)$.
\end{itemize} 

\begin{prop} \label{prop21}
Let $q\in \K^3$. There exists a linear endomorphism $\varphi_q$ of $\WQSym$ such that $\varphi_q\circ L=\Gamma_q$ if, and only if,
$q=(1,0,0)$ or $q=(0,1,0)$. 
\end{prop}

\begin{proof} $\Longrightarrow$. If $\varphi_q$ exists, then:
\begin{align*}
\varphi_q((123))&=\varphi_q \circ L(\tdtroisdeux{$1$}{$2$}{$3$})=\Gamma_q(\tdtroisdeux{$1$}{$2$}{$3$})=
(123)+q_1(112)+q_1(122)+q_1^3(111),\\
\varphi_q((213))&=\varphi_q \circ L(\tdtroisdeux{$2$}{$1$}{$3$})=\Gamma_q(\tdtroisdeux{$2$}{$1$}{$3$})=
(213)+q_2(112)+q_1(212)+q_1^2q_2(111),\\
\varphi_q((312))&=\varphi_q \circ L(\tdtroisdeux{$2$}{$3$}{$1$})=\Gamma_q(\tdtroisdeux{$2$}{$3$}{$1$})=
(312)+q_1(211)+q_2(212)+q_1q_2^2(111).
\end{align*}
Moreover:
\begin{align*}
\varphi_q\circ L(\tdun{$1$}\tddeux{$2$}{$3$})&=\varphi_q((123)+(213)+(312))\\
&=(123)+(213)+(312)+(q_1+q_2)(112)+q_1(211)+q_1(122)+(q_1+q_2)(212)\\
&+(q_1^3+q_1^2q_2+q_1q_2^2)(111)\\
&=\Gamma_q(\tdun{$1$}\tddeux{$2$}{$3$})\\
&=(123)+(213)+(312)+(112)+q_1(211)+q_1(122)+(212)+q_1(111).
\end{align*}
Identifying these two expressions, the coefficient of $(112)$ gives $q_1+q_2=1$; the coefficient of $(111)$ gives:
\begin{align*}
q_1=q_1^3+q_1^2q_2+q_2q_1^2&\Longleftrightarrow q_1(q_1^2+q_1q_2+q_2^2-1)=0\\
&\Longleftrightarrow q_1((q_1+q_2)^2-q_1q_2-1)=0\\
&\Longleftrightarrow q_1^2q_2=0.
\end{align*}
So $(q_1,q_2)=(1,0)$ or $(0,1)$. Moreover:
\begin{align*}
\varphi_q((121))&=\varphi_q\circ L(\tddeux{$1,3$}{$2$}\hspace{2mm})=\Gamma_q(\tddeux{$1,3$}{$2$}\hspace{2mm})=(121),\\
\varphi_q((212))&=\varphi_q\circ L(\tddeux{$2$}{$1,3$}\hspace{2mm})=\Gamma_q(\tddeux{$2$}{$1,3$}\hspace{2mm})=(212),
\end{align*}
and:
$$(121)+(212)=\varphi_q((121)+(212))=\varphi_q\circ L(\tdun{$1,3$}\hspace{3mm}\tdun{$2$})
=\Gamma_q(\tdun{$1,3$}\hspace{3mm}\tdun{$2$})=(121)+(212)+q_3(111).$$
The coefficient of $(111)$ gives $q_3=0$. So $q=(1,0,0)$ or $(0,1,0)$. \\

$\Longleftarrow$. This comes from proposition \ref{prop19}. \end{proof}

\subsection{Links with special posets}

If $\T\in \TT$ is $T_0$, then all its linear extensions are permutations. Seeing $\FQSym$ as a Hopf subalgebra of $(\WQSym,\shuffle ,\Delta)$, 
by restriction we obtain a Hopf algebra morphism $L:\h_\SP\longrightarrow \FQSym$,
 which is the morphism $L$ defined in section \ref{section1.2}.
 
If $\T\in \TT$ is not $T_0$, then no generalized T-partition of $\T$ is a permutation, so $\varpi \circ \Gamma_q(\T)=0$.
Hence, we have a commutative diagram of Hopf algebras:
$$\xymatrix{\h_\TT\ar@{->>}[d] _{\theta_0}\ar[r]^{\Gamma_q}&\WQSym\ar@{->>}[d]^{\varpi}\\
\h_\SP\ar[r]_L&\FQSym}$$

If $\T\in \TT$ is $T_0$, then a T-partition of $\T$ is a P-partition of the poset associated to $\T$, in Stanley's sense \cite{Stanley}, 
and we obtain the commutative diagram:
$$\xymatrix{&\FQSym
\ar[d]^{\theta_{(1,0,0)}}\\
\h_\SP\ar[r]_\Gamma \ar[ru]^L&\WQSym}$$

\subsection{The order on packed words}

\label{section4.5} Let us precise the properties of the order on packed words of definition \ref{defi17}. 
We shall use the following notion, which, in some sense, is the generalization of the notion of ascents of permutations.

\begin{defi}
Let $f$ be a packed word of length $n$, and let $i\in [n]$. We shall say that $i\in M(f)$ if:
\begin{itemize}
\item $f(i)<\max(f)$.
\item For all $j\in [n]$, $f(j)=f(i)$ $\Longrightarrow$  $j\leq i$.
\item For all $j\in [n]$, $f(j)=f(i)+1$ $\Longrightarrow$ $j>i$.
\end{itemize}\end{defi}

{\bf Example.} If $f=(412133)$, then $M(f)=\{3\}$. \\

{\bf Remark.} If $f$ is a permutation, then $i\in M(f)$ if, and only if, $f^{-1}(f(i)+1)>i$. \\

The aim of this section is to prove the following theorem:

\begin{theo}\label{theo23}
\begin{enumerate}
\item Let $f,g$ be two packed words. Then $f\leq g$ if, and only if, $Std(f)=Std(g)$ and $M(f)\subseteq M(g)$. 
\item For all $n \geq 1$, there is an isomorphism of posets:
$$\Phi:\left\{\begin{array}{rcl}
(\{\mbox{packed words of length $n$}\},\leq)&\longrightarrow&\displaystyle \bigsqcup_{\sigma \in \mathfrak{S}_n} 
(\{\mbox{subsets of $M(\sigma)$ }\},\subseteq)\\
f&\longrightarrow&M(f)\subseteq M(Std(f)).
\end{array}\right.$$
Hence, for all $n\geq 1$, the poset of packed words of length $n$ is a disjoint union of posets, indexed by $\S_n$, the part indexed by $\sigma$
being isomorphic to the poset of subsets of $M(\sigma)$, partially ordered by the inclusion.
\end{enumerate}\end{theo}

\begin{lemma} \label{lemma24}\begin{enumerate}
\item For any packed word $g$, $g\leq Std(g)$.
\item Let $f,g$ be packed words. If $f\leq g$, then $Std(f)=Std(g)$.
\end{enumerate}\end{lemma}

\begin{proof} We put $\sigma=Std(f)$, $\tau=Std(g)$ and, for all $p\in [\max(g)]$, $g^{-1}(\{p\})=C_p$. 

1. Let $i,j \in [n]$.  We assume that $i\in C_p$ and $j\in C_q$. 
If $\tau(i)\leq \tau(j)$, by definition of the standardization, $p\leq q$, so $g(i)=p\leq q=g(j)$. 
If $\tau(i)=\tau(j)$, as $\tau$ is a permutation, $i=j$, and $g(i)=g(j)$. If $\tau(i)<\tau(j)$ and $j>i$, then $p\leq q$.
As $\tau$ is increasing on $C_p$ by definition of the standardization, $p=q$ is impossible. So $p<q$, and $g(i)=p<q=g(j)$.
We obtain $g\leq \tau$. \\

2. As $f\leq g$, $f$ is constant on $C_p$ for all $p$. We put $f(C_p)=\{c_p\}$. If $i\in C_p$ and $j\in C_q$, with $p<q$,
then $g(i)=p\leq q=g(j)$, so $f(i)\leq f(j)$: $c_p\leq c_q$. 
If $c_p=c_q$ and $p<q$, let $i\in C_p$ and $j\in C_q$. If $j<i$, as $g(i)=p<q=g(j)$ and $f\leq g$, $c_p=f(i)<f(j)=c_q$: contradiction.
Hence, if $c_p=c_q$, for all $i\in C_p$, for all $j\in C_q$, $i<j$, which is shortly denoted by $C_p<C_q$. 

As $f$ is constant on $C_p$, $\sigma$ is increasing on $C_p$.
If $p<q$ and $c_p\neq c_q$, then $c_p<c_q$. By definition of the standardization, for all $i\in C_p$, $j\in C_q$,
$\sigma(i)<\sigma(j)$. If $p<q$ and $c_p=c_q$, then for all $i\in C_p$, $j\in C_q$, $i<j$. As $f$ is constant on $C_p\sqcup C_q$,
$\sigma(i)<\sigma(j)$. Finally:
\begin{itemize}
\item $\sigma$ is increasing on $C_p$ for all $p$.
\item If $p<q$, $i\in C_p$ and $j\in C_q$, $\sigma(i)<\sigma(j)$.
\end{itemize}
So $\sigma=Std(g)$. \end{proof}

\begin{lemma} \label{lemma25}
Let $\sigma \in \mathfrak{S}_n$, $n\geq 1$. The following map is bijective:
$$\phi_\sigma:\left\{\begin{array}{rcl}
\{\mbox{$f$ packed word}\mid Std(f)=\sigma\}&\longrightarrow&\{I\mid I\subseteq M(\sigma)\}\\
f&\longrightarrow&M(f).
\end{array}\right.$$
\end{lemma}

\begin{proof} Let $f$ be a packed word such that $Std(f)=\sigma$.  We put $f^{-1}(\{p\})=C_p$ for all $p\in [\max(f)]$.
Let $i \in M(f)$. Assume that $i\in C_p$. Then $p<\max(f)$, $i$ is the greatest element of $C_p$, and, if $j$ is the smallest element of $C_{p+1}$,
$i<j$.  By definition of the standardization, $\sigma(j)=\sigma(i)+1$. As $j>i$, $i\in M(\sigma)$, so $M(f)\subseteq M(\sigma)$,
and $\phi_\sigma$ is well-defined. \\

We define a map $\psi_\sigma:\{I\mid I\subseteq M(\sigma)\}\longrightarrow \{\mbox{$f$ packed word}\mid Std(f)=\sigma\}$
in the following way. If $I \subseteq M(\sigma)$, we define $f(\sigma^{-1}(i))$ by induction: $f(\sigma^{-1}(1))=1$, and, for all $i\in[n-1]$:
\begin{itemize}
\item If $\sigma^{-1}(i)\in I$ or if $\sigma^{-1}(i)\notin M(\sigma)$, then $f(\sigma^{-1}(i+1))=f(\sigma^{-1}(i))+1$.
\item If $\sigma^{-1}(i)\in M(\sigma)\setminus I$, $f(\sigma^{-1}(i+1))=f(\sigma^{-1}(i))$.
\end{itemize}
Clearly, $f$ is a packed word. Let us prove that $Std(f)=\sigma$. For all $p\in [\max(f)]$, we put $f^{-1}(\{p\})=C_p$. 
By definition of $f$, for all $p$, there exist $i_p\leq j_p$ such that $C_p=\sigma^{-1}(\{i_p,\ldots,j_p\})$,
and $\sigma^{-1}(i_p),\ldots,\sigma^{-1}(j_p-1)\in M(\sigma)$, which implies:
$$\sigma^{-1}(i_p)<\sigma^{-1}(i_p+1)<\ldots<\sigma^{-1}(j_p-1)<\sigma^{-1}(j_p).$$
We obtain that $\sigma^{-1}$ is increasing on $i_p,\ldots,j_p$, so $\sigma$ is increasing on $C_p$. 
Moreover, if $p<q$, $i\in C_p$ and $j\in C_q$, by definition of $f$, putting $\sigma^{-1}(i)=k$ and $\sigma^{-i}(j)=l$, $k<l$. 
Consequently, $\sigma(i)=k<l=\sigma(j)$. We obtain that $Std(f)=\sigma$. We can put $\psi_\sigma(I)=f$,
and then $\psi_\sigma$ is a well-defined map. \\

Let $I\subseteq M(\sigma)$, and $f=\psi_\sigma(I)$. For all $p\in [\max(f)]$, we put $f^{-1}(\{p\})=C_p$. 
If $i\in M(f)$, then $i$ is the greatest element of $C_p$ with $p=f(i)<\max(f)$, and if $j$ is the smallest element of $C_{p+1}$, then $i<j$. 
As $\sigma=Std(f)$, $\sigma(j)=\sigma(i)+1$, so $i\in M(\sigma)$, and $M(f) \subseteq M(\sigma)$. 
By definition of $f$, $i\in I$ or $i\notin M(\sigma)$, so $i\in I$: $M(f)\subseteq I$. Let $i\in I$. We put $k=\sigma(i)$.
Then $f(\sigma^{-1}(k+1))=f(\sigma^{-1}(k))+1$. By definition of $f$, $i$ is the greatest element of $C_p$, with $p=f(i)$ 
and $j=\sigma^{-1}(k+1)$ is the smallest element of $C_{p+1}$. As $\sigma^{-1}(k)=i\in M(\sigma)$,
$\sigma^{-1}(k+1)=j>\sigma^{-1}(k)=i$: $i\in M(f)$. We obtain $M(f)=I$, that is to say $\phi_\sigma\circ \psi_\sigma(I)=I$. \\

Let $f$ be a packed word such that $Std(f)=\sigma$, $I=M(f)$ and $g=\psi_\sigma(I)$. 
For all $p\in [\max(f)]$, we put $f^{-1}(\{p\})=C_p$. Let us prove that $f(\sigma^{-1}(i))=g(\sigma^{-1}(i))$ for all $i$ by induction. 
If $i=1$, as $\sigma=Std(f)$, $f(\sigma^{-1}(1))=1=g(\sigma^{-1}(1))$. Let us assume that $f(\sigma^{-1}(i))=g(\sigma^{-1}(i))$. 
We obtain three different cases.
\begin{enumerate}
\item If $\sigma^{-1}(i)\in I$, then $\sigma^{-1}(i)$ is the greatest element of $C_p$,
with $p=f(\sigma^{-1}(i))$, and if $j$ is the smallest element of $C_{p+1}$, then $i<j$. As $Std(f)=\sigma$,
$j=\sigma^{-1}(\sigma(\sigma^{-1}(i))+1)=\sigma^{-1}(i+1)$, and $f(\sigma^{-1}(i+1))=p+1=
f(\sigma^{-1}(i))+1=g(\sigma^{-1}(i+1))$.
\item If $\sigma^{-1}(i)\notin M(\sigma)$, then $\sigma^{-1}(i+1)<\sigma^{-1}(i)$. As $\sigma=Std(f)$, 
necessarily $\sigma^{-1}(i)$ is the greatest element of $C_p$ and $\sigma^{-1}(i+1)$ is the smallest element of $C_{p+1}$.
We obtain $f(\sigma^{-1}(i+1))=p+1=f(\sigma^{-1}(i))+1=g(\sigma^{-1}(i+1))$.
\item If $\sigma^{-1}(i)\in M(\sigma)\setminus I$, then $\sigma^{-1}(i)<\sigma^{-1}(i+1)$.
As $\sigma=Std(f)$ and $i\notin I$, $\sigma^{-1}(i)$ and $\sigma^{-1}(i+1)$ are in the same $C_p$,
so $f(\sigma^{-1}(i+1))=p=f(\sigma^{-1}(i))=g(\sigma^{-1}(i+1))$.
\end{enumerate}
As a conclusion, $g=f$, so $\psi_\sigma \circ \phi_\sigma(f)=f$. \end{proof}

\begin{proof} (theorem \ref{theo23}).

1. $\Longrightarrow$. If $f\leq g$, by lemma \ref{lemma24}, $Std(f)=Std(g)$. We denote by $\sigma$ this permutation. 
If $I=M(f)$ and $J=M(g)$, then $f=\psi_\sigma(I)$ and $g=\psi_\sigma(J)$. For all $p\in [\max(f)]$, we put $f^{-1}(\{p\})=C_p$.
For all $q\in [\max(g)]$, we put $g^{-1}(\{q\})=C'_q$. \\

Let $k\in I$. We put $\sigma(k)=i$. By construction of $\psi_\sigma(I)$, $k=\sigma^{-1}(i)$ is the greatest letter of $C_p$
for $p=f(k)$, and if $l=\sigma^{-1}(i+1)$ is the smallest letter of $C_{p+1}$, then $k<l$. 
Consequently, if $k'\in C_p$, $l'\in C_{p+1}$, then $k'\leq k<l\leq l'$. If $g(k')\geq g(l')$, as $f\leq g$, we should have $f(k')>f(l')$:
this is a contradiction, as $f(k')=p$ and $f(l')=p+1$. So $g(k')<g(l')$. 
Moreover, $f$ is constant on $C'_q$ for all $q$, as $f\leq g$. If $k\in C'_q$, then $C'_q\subseteq C_p$ with $p=f(k)$.
Moreover, $l\in C_{p+1}$, so $l\notin C'_q$. As $Std(g)=\sigma$, $l=\sigma^{-1}(i+1) \in C'_{q+1}$, which implies $C'_{q+1} \subseteq C_{p+1}$.
So for all $k'\in C_q$, $l'\in C_{q+1}$, $k'<l'$: $k\in M(g)=J$, and $I\subseteq J$. \\

1. $\Longleftarrow$. We put $I=M(f)$, $J=M(g)$, such that $f=\psi_\sigma(I)$ and $g=\psi_\sigma(J)$, with $\sigma=Std(f)=Std(g)$.
\begin{itemize}
\item As $I\subseteq J$, the change of values of $f$ in the definition of $\psi_\sigma(I)$ are also change of values of $g$ in the definition of
$\psi_\sigma(J)$; consequently, if $g(i)=g(j)$, then $f(i)=f(j)$. 
\item If $g(k)\leq g(l)$, we put $\sigma(k)=i$ and $\sigma(l)=j$. By construction of $\psi_\sigma(J)$, $i<j$.
By construction of $\psi_\sigma(I)$, $f(k)=f(\sigma^{-1}(i))=\leq f(\sigma^{-1}(j))=f(l)$.
\item If $g(k)<g(l)$ and $k>l$, we put $\sigma(k)=i$ and $\sigma(l)=j$. Then the interval $\{i,\ldots, j-1\}$
contains an element which does not belong to $M(\sigma)$ (otherwise, it would contain only elements of $M(\sigma)$, and then $k\leq l$). 
By definition of $\psi_\sigma(I),$ $f(k)<f(l)$.
\end{itemize}
Finally, $f\leq g$.\\

2. For all $\sigma\in \mathfrak{S}_n$, $\phi_\sigma$ is bijective: this implies that $\Phi$ is bijective. 
By the first point, $\Phi$ is an isomorphism of posets.  \end{proof}

{\bf Remark.} In particular, if $\sigma$ is a permutation, $\varphi_{(1,0,0)}(\sigma)$ is the sum of all packed words $f$ such that $Pack(f)=\sigma$.
This implies that the restriction of $\varphi_{(1,0,0)}$ to $\FQSym$ is the map $\varphi$ defined in section \ref{section1.2}.

\begin{cor}\label{cor26}
For all $n \geq 1$, for all $\sigma \in \S_n$:
$$\sharp\{\mbox{$w$ packed word of length }n\mid Std(w)=\sigma\}=2^{\sharp M(\sigma)}.$$
For all $n \geq 1$:
$$\sharp\{\mbox{packed words of length }n\}=\sum_{\sigma\in \mathfrak{S}_n} 2^{\sharp M(\sigma)}.$$
\end{cor}

\bibliographystyle{amsplain}
\bibliography{biblio2}

\end{document}

%% file: arbres.tex
\newcommand{\tdun}[1]{\begin{picture}(10,5)(-2,-1)
\put(0,0){\circle*{2}}
\put(3,-2){\tiny #1}
\end{picture}}

\newcommand{\tddeux}[2]{\begin{picture}(12,10)(0,-1)
\put(3,0){\circle*{2}}
\put(3,5){\circle*{2}}
\put(3,0){\line(0,1){5}}
\put(6,-2){\tiny #1}
\put(6,3){\tiny #2}
\end{picture}}

\newcommand{\tdtroisun}[3]{\begin{picture}(20,12)(-5,-1)
\put(3,0){\circle*{2}}
\put(6,7){\circle*{2}}
\put(0,7){\circle*{2}}
\put(-0.65,0){$\vee$}
\put(5,-2){\tiny #1}
\put(9,5){\tiny #2}
\put(-5,5){\tiny #3}
\end{picture}}

\newcommand{\tdtroisdeux}[3]{\begin{picture}(12,15)(-2,-1)
\put(0,0){\circle*{2}}
\put(0,5){\circle*{2}}
\put(0,10){\circle*{2}}
\put(0,0){\line(0,1){5}}
\put(0,5){\line(0,1){5}}
\put(3,-2){\tiny #1}
\put(3,3){\tiny #2}
\put(3,9){\tiny #3}
\end{picture}}

\newcommand{\tdquatredeux}[4]{\begin{picture}(20,20)(-5,-1)
\put(3,0){\circle*{2}}
\put(6,7){\circle*{2}}
\put(0,7){\circle*{2}}
\put(0,14){\circle*{2}}
\put(-.65,0){$\vee$}
\put(0,7){\line(0,1){7}}
\put(5,-2){\tiny #1}
\put(9,5){\tiny #2}
\put(-5,5){\tiny #3}
\put(-5,12){\tiny #4}
\end{picture}}

%% file: posets.tex
\newcommand{\pdtroisun}[3]{\begin{picture}(20,12)(-5,-1)
\put(3,7){\circle*{2}}
\put(-0.65,0){$\wedge$}
\put(6,0){\circle*{2}}
\put(0,0){\circle*{2}}
\put(5,5){\tiny #1}
\put(-5,-2){\tiny #2}
\put(9,-2){\tiny #3}
\end{picture}}

\newcommand{\pdcinq}[5]{\begin{picture}(20,23)(-5,-1)
\put(3,0){\circle*{2}}
\put(-0.65,0){$\vee$}
\put(6,7){\circle*{2}}
\put(0,7){\circle*{2}}
\put(3,14){\circle*{2}}
\put(3,21){\circle*{2}}
\put(3,14){\line(0,1){7}}
\put(-0.65,7){$\wedge$}
\put(5,-2){\tiny #1}
\put(9,5){\tiny #2}
\put(-5,5){\tiny #3}
\put(5,12){\tiny #4}
\put(5,19){\tiny #5}
\end{picture}}